\documentclass[12pt,a4paper]{amsart}
\usepackage[latin1]{inputenc}   
\usepackage{ dsfont }

\usepackage{amsmath}
\usepackage{amsfonts}
\usepackage{amssymb}
\usepackage{esint}
\usepackage{graphicx}
\usepackage{amsthm}
\usepackage{enumerate}
\usepackage{verbatim}
\usepackage{hyperref}
\usepackage[english]{babel}
\usepackage{bbm}
\usepackage{mathrsfs}
\usepackage{relsize}
\usepackage{textcomp}

\def \N{\mathbb{N}}
\def \R{\mathbb{R}}
\def \E{\mathbb{E}}
\def \P{\mathbb{P}}

\theoremstyle{plain} 
\newtheorem{thm}{Theorem}[section] 
\newtheorem{cor}[thm]{Corollary} 
\newtheorem{lem}[thm]{Lemma} 
\newtheorem{prop}[thm]{Proposition} 

\newtheorem{defn}[thm]{Definition}
 
\newtheorem{oss}[thm]{Remark}
\theoremstyle{definition} 
\newtheorem{rem}[thm]{Remark}
\numberwithin{equation}{section}

\newcommand{\be}{\begin{equation}}
\newcommand{\ee}{\end{equation}}

\def\rife#1{(\ref{#1})}
\def\m{\noalign{\medskip}}
\def\vfi{\varphi}

\def\dive{{\rm div}}
\def\de{\delta}

\def\ga{\gamma}
\def\vep{\varepsilon}
\def\elle#1{L^{#1}(\Omega)}
\def\parelle#1{L^{#1}(Q_T)}

\newcommand{\miezz}{\frac{1}{2}}

\newcommand{\eps}{\varepsilon}

\newcommand{\into}{\ensuremath{\int_{\Omega}}}
\newcommand{\intoe}{\ensuremath{\int_{\Omega_\eps}}}

\newcommand{\intie}{\ensuremath{\int_{0}^{t}\int_{\Omega_\eps}}}
\newcommand{\intif}{\ensuremath{\int_{0}^{T}\int_{\Omega}}}

\newcommand{\norm}[1]{\ensuremath{\left\Arrowvert #1 \right\Arrowvert}}

\newcommand{\spazio}{\hspace{0.08cm}}

\newcommand{\intok}[1]{\int_{D_{#1}}}
\newcommand{\intifk}[1]{\int_{0}^{T}\int_{D_{#1}}}

\newcommand{\intfk}[1]{\int_{t}^{T}\int_{D_{#1}}}

\newcommand{\etd}{e^{\theta(d)}}

\newcommand{\conc}{\left(\delta+\frac{\psi(\bar{x},\bar{y},\bar{z})}{\delta}\right)}
\newcommand{\con}{\left(\delta+\frac{\psi({x},{y},{z})}{\delta}\right)}
\newcommand{\bx}{\bar{x}}
\newcommand{\by}{\bar{y}}
\newcommand{\bz}{\bar{z}}
\newcommand{\bt}{\bar{t}}

\begin{document}

\title{Mean field games under invariance conditions for the state space}

\author{Alessio Porretta}\thanks{Dipartimento di Matematica, Universit\`a di Roma Tor Vergata. 
Via della Ricerca Scientifica 1, 00133 Roma, Italy. \texttt{porretta@mat.uniroma2.it}} 

\author{Michele Ricciardi}\thanks{Dipartimento di Matematica, Universit\`a di Roma Tor Vergata. 
Via della Ricerca Scientifica 1, 00133 Roma, Italy. \texttt{ricciard@mat.uniroma2.it}}

\date{\today}

\maketitle

\begin{abstract}
We investigate mean field game systems under invariance conditions for the state space, otherwise called {\it viability conditions} for the controlled dynamics. First we analyze separately the Hamilton-Jacobi and the Fokker-Planck equations, showing how the invariance condition on the underlying dynamics yields the existence and uniqueness, respectively in $L^\infty$ and in $L^1$. Then we apply this analysis to mean field games.
We investigate further the regularity of solutions proving, under some extra conditions, that the value function is (globally) Lipschitz and semiconcave. This latter regularity  eventually leads the distribution density to be bounded, under suitable conditions. The results are not restricted to smooth domains.
\end{abstract}

\section{Introduction}

The theory of mean field games  was introduced by J.M. Lasry and P.L. Lions (\cite{LL1}, \cite{LL2}, \cite{LL-japan}) in order to describe Nash equilibria in differential games with 
infinitely many (small and undistinguishable) agents, using tools from mean-field theories. A similar notion of Nash equilibria was also developed in the same years by P. Caines, M. Huang and R. Malham\'e \cite{HCM}. The macroscopic description used in mean field game theory leads to study  coupled systems of PDEs, where the Hamilton-Jacobi-Bellman equation satisfied by  the single agent's value function is coupled with the Kolmogorov Fokker-Planck equation satisfied by the distribution law of the population. The simplest form of this system is the following
\begin{equation}\label{mfg}
\begin{cases}
-\partial_t u - \sum\limits_{i,j} a_{ij}(x) \partial_{ij}^2 u +H(t,x,Du)=F(t,x,m)\,, \hspace{2cm}(t,x)\in (0,T)\times\Omega\\
\partial_t m - \sum\limits_{i,j} \partial_{ij}^2 (a_{ij}(x)m) -\mathrm{div}(mH_p(t,x,Du))=0\,,\hspace{2cm}(t,x)\in (0,T)\times\Omega\\
m(0)=m_0 \hspace{2cm} u(T)=G(x,m(T))
\end{cases}\end{equation}
where $\partial_{ij}^2 (\cdot)= \frac{\partial^2 (\cdot) }{\partial x_i\partial x_j}$ denotes a  second order partial differentiation and $\dive(\cdot) $ is the usual divergence operator, and where $H_p(t,x,p)$ denotes $\frac{\partial H(t,x,p)}{\partial p}$, for $p\in \R^N$.

Here $\Omega\subseteq\mathbb{R}^N$ is an open and bounded set and $x\in \Omega$ represents  the dynamical state of the generic agent, $m(t)$ is the distribution law of the agents at time $t$ (and $m(t,x)$ denotes its density, if $m(t)\in L^1$), $F(t,x,m)$, $G(x,m(T))$ are respectively a  running cost and a final pay-off and $H(t,x,Du)$ is the Hamiltonian function associated to the cost of dynamic control of the individuals. More details on the  interpretation of solutions in terms of stochastic control  will be given later.  

There is by now an extensive literature concerning mean field game systems of this kind, and many fundamental issues have been discussed so far such as existence or nonexistence, regularity and uniqueness of solutions,  long time behavior etc...However, most of the literature considers  the case that the state variable $x$ belongs to the flat torus (i.e. solutions are periodic). 
But in many applied models, boundary conditions turn out to be  a  crucial issue. A significant case occurs when the dynamical state needs to remain inside some given domain of existence, say if some natural restriction needs to be preserved. For instance,  in many models appearing in economics,  a scalar state variable needs to remain above or below given thresholds (e.g.  if $x$ denotes a stock quantity, or the reserve of  a fossil fuel, or  a wealth level, see  models described in \cite{Buera+al}, \cite{GLL}). 

There are two typical ways in which the proposed models   handle this kind of situation: either one considers the {\it state constraint} control problem, in which case one uses the control in order to satisfy the required restriction, or alternatively the drift-diffusion terms are built in the model so that the state does not leave the domain, regardless of the control.
This latter situation is what we are going to study in this paper. Namely, we assume that the state variable $x$ belongs to a  bounded domain $\Omega \subset \R^d$ and we will assume structure conditions, on the diffusion and the Hamiltonian terms, which imply that the domain $\Omega$ is an invariant set for the underlying controlled dynamics, and this invariance occurs {\it for any choice of the control}. In the control community, sometimes this property is referred to as the {\it viability of the state space}.  

Let us stress that considering  the domain to be  invariant for all controls is different from considering the state constraint control problem; in very rough words, one can say that the viability of the state space plays like a regularizing condition of the underlying dynamics, whereas the state constraint problem leads to formation of singularities at the boundary, because of the forced action of the control.  

In the case of uncontrolled SDEs
\be\label{dyn}
dX_t= b(t,X_t) \, dt + \sqrt 2\, \sigma(X_t)\, dW_t \,,\quad X_0=x \in \Omega 
\ee
the conditions on the coefficients  $b$ and $\sigma$ which let $\Omega$ be an invariant set are extensively discussed in the literature, at least in the case that $\sigma$ and $b$ are globally Lipschitz. We refer the reader to \cite{CDF} and the literature therein. The case that the diffusion is controlled, so $b= b(t,x,\alpha)$, $\sigma= \sigma(t,x,\alpha)$ for $\alpha$ in a set of controls, was recently discussed in \cite{BCR}, \cite{CCR} in terms of (viscosity solutions) of the associated Bellman operators. In our study, we let aside by now the possibility that the diffusion part is controlled;  on the other hand, we aim at giving general conditions on the diffusion matrix and drift terms which may apply to both degenerate and non degenerate operators, to possibly unbounded drifts and possibly non Lipschitz matrix $\sigma$. For the  dynamics
\rife{dyn} in a $C^2$ domain $\Omega$, we formulate this invariance condition by requiring that the following inequality holds in a neighborhood of the boundary:
\be\label{inv-lin}
{\rm tr}(a(x) D^2 d(x)) + b(t,x)\cdot Dd(x) \geq \frac{a(x)Dd(x)\cdot Dd(x)}{d(x)} - C \, d(x)
\ee
for some constant $C>0$, where $d(x)$ is the distance function to the boundary, $a(x)=(\sigma\,\sigma^*)(x)$ is the diffusion matrix and ${\rm tr}(\cdot)$ is the trace operator. This condition reduces to the well known necessary and sufficient condition (\cite{CDF}) if $\sigma$ is Lipschitz continuous and $\sigma^*(x) Dd(x)=0$ on the boundary.  However, it also includes more general cases,  like $\sigma$ being only $1/2$-H\"older continuous (to this respect, it generalizes the condition used in  \cite{BCR} for H\"older coefficients) and, last but not least, it also applies to the case of non degeneracy ($a(x)$ coercive up to the boundary) if the drift $b(x)$ is allowed to be unbounded (as in \cite{LP}).

In case of Bellman operators, say if the dynamics is controlled,  the viability of the state space is ensured provided the same condition as \rife{inv-lin} is required to hold for all controls.  More specific examples will be given in the next Section. 
We point out that this kind of condition is also related to the notion of characteristic points of the boundary, namely to the question whether boundary conditions should be prescribed or not for the corresponding Bellman operator, see e.g. \cite{BB}, \cite{BR} and \cite{Fi} for the linear case.  
\vskip0.5em
For a general PDE approach to the Hamilton-Jacobi equation  
\be\label{HJ}
-\partial_t u - \sum\limits_{i,j} a_{ij}(x) \partial_{ij}^2 u +H(t,x,Du)=0\,,\qquad (t,x)\in (0,T)\times \Omega
\ee
we replace the invariance condition on the dynamics with a structure condition formulated directly on the Hamiltonian function. Namely, in the same spirit as above, we require that the diffusion matrix $a(x)$ and the Hamiltonian $H(t,x,p)$ satisfy the inequality
\be\label{inv-hjb}
{\rm Tr}(a(x) D^2 d(x)) - H_p(t,x,p)\cdot Dd(x) \geq \frac{a(x)Dd(x)\cdot Dd(x)}{d(x)} - C \, d(x)\, \qquad \forall p\in \R^N
\ee
for some constant $C>0$, for $x$ in a neighborhood of the boundary and $t\in (0,T)$. 

As it is intrinsic to mean field games, we are going to study not only the properties of the HJB equation under condition \rife{inv-hjb} but also the properties of the Kolmogorov equation which appear, roughly speaking, in a  dual form. The key point, as we will see in our results, is that the {\it invariance condition} ensures that, on one hand, the uniqueness holds just in the class of bounded solutions for HJB, on another hand a global $L^1$- stability holds for the KFP equation.

In the end, the contribution of this article will be the analysis of HJB equations, Fokker-Planck equations and, eventually, mean field games under the invariance structure conditions formulated above. In order to focus on the boundary behavior, we assume throughout  the article that the matrix $a(x)$ is Lipschitz continuous in $\Omega$ and is elliptic in the interior of $\Omega$, namely
\be\label{ell0}
a(x) \ge 0\quad \forall\ x\in\overline{\Omega},  \quad \hbox{and\ \  $a(x) >0$  if $x\in\Omega$.}
\ee
This latter condition avoids the superposition of interior and boundary degeneracy which would make the analysis more complicated. Besides, the interior ellipticity will guarantee local compactness of the solutions which allows us to use a standard framework of weak (distributional) solutions. While we defer a precise statement of our results to the next sections, we list here a short summary of what we prove in this paper:
\begin{itemize}

\item[(a)] Assuming the structure conditions \rife{inv-hjb} and \rife{ell0}, and requiring the Hamiltonian $H(t,x,p)$ to be convex in $p$, locally bounded in $x$ with at most quadratic growth with respect to $p$ and such that $H(t,x,0)$ is globally bounded, we show existence and uniqueness of bounded solutions to the HJB equation \rife{HJ} with bounded terminal pay-off.

\item[(b)] Assuming that the drift $b(t,x)$ is locally bounded and the structure conditions \rife{inv-lin} and \rife{ell0}, we prove that, for  any initial probability density $m_0\in \elle1$  the Fokker-Planck equation
$$
\begin{cases}
\partial_t m - \sum\limits_{i,j} \partial_{ij}^2 (a_{ij}(x)m) +\mathrm{div}(m\, b(t,x))=0\,, \qquad(t,x)\in (0,T)\times\Omega\\
m(0)=m_0 
\end{cases}
$$
admits  a unique weak solution $m\in C^0([0,T];L^1(\Omega))$. Here by a  weak solution we mean that  
$$
\intif m\, {\mathcal L}(\phi) dxdt =\into m_0\phi(0)dx
\qquad 
$$
for every $\phi\in C^0([0,T];L^1(\Omega))\cap L^\infty$ such that ${\mathcal L}(\phi)\in L^\infty$, where ${\mathcal L}(\phi)= -\partial_t \phi - \sum\limits_{i,j} a_{ij}(x) \partial_{ij}^2 \phi - b(t,x)\cdot D\phi$.

\item[(c)] Under the conditions on $a(x)$, $H$ assumed in item (a) (in particular, under the invariance condition \rife{inv-hjb}), and assuming that the coupling terms $F(t,x,m)$ and $G(x,m)$ satisfy global  bounds in $L^\infty$ and suitable continuity conditions with respect to $m$, we prove that the mean field game system \rife{mfg} admits a weak solution, where the two equations are formulated in the sense specified, respectively, by previous results in (a) and (b). This kind of solution of \rife{mfg} is also unique under usual monotonicity conditions upon $F$ and $G$ (with respect to $m$). 

\item[(d)] Assuming in addition that $a(x)= (\sigma\sigma^*)(x)$ with $\sigma$ Lipschitz,  plus a few natural structure conditions on the Hamiltonian $H$ and further regularity of $F$ and $G$, we prove additional regularity for the solution $(u,m)$ of \rife{mfg}; namely, that $u$ is (globally) Lipschitz continuous and semi concave in $\Omega$. Moreover, in this case $m$ is (globally) bounded provided  $\sum_{i,j}[ \frac{\partial a_{ij}}{\partial x_i}+ H_{p_j}(t,x,Du)]\,\nu_j \geq 0$ on the boundary, where $\nu$ is the outward unit normal.

\end{itemize}

The spirit of the above results is that the invariance condition plays like a  (transparent)  boundary condition for the two equations. In fact,  the existence
of solutions will be provided by limit of (penalized) standard Neumann problems. The stochastic interpretation of the invariance condition easily explains that  a kind of (transparent and soft) reflection naturally occurs near the boundary.  
The regularity results mentioned in item (d) show how the invariance condition may prevent the formation of singularities which, conversely, would occur in case of the state constraint problem. Indeed, global semi concavity may be lost in that case, see e.g. the recent paper \cite{CCC}.

Last but not least, we generalize our results to possibly non smooth domains. This generalization includes in particular the case that $\Omega= \prod_{i=1}^N (a_i,b_i)$ is a $N$-dimensional rectangle, which is often the case in applications.

\vskip1em

We conclude by summarizing the organization of the paper. In Section 2 we list the main notation and the standing assumptions which hold throughout the paper; then we give a  few examples of control problems which fit our conditions and we properly state the main existence and uniqueness results which are proved.
Section 3 is devoted to the study of the single HJB equation \rife{HJ} under the invariance condition. Section 4 is devoted to the analysis of the single Fokker-Planck equation in the same context. The mean field game system \rife{mfg} is studied and characterized in Section 5.  Section 6 contains the additional regularity results on the solutions and specifically the Lipschitz and semi concavity regularity; at this stage we need to make additional assumptions on the nonlinearity and this is why those results are not mentioned earlier in  Section 2.  Finally, Section 7 contains the generalization to non smooth domains.
We leave to the Appendix the proof of a couple of technical results.

\section{Preliminaries: assumptions and examples}
	
We assume throughout the paper that $\Omega$ is a bounded open subset of $\R^N$, $N\geq 1$.  We denote $Q_T:= (0,T)\times \Omega$. We recall that the \textit{oriented distance} from $\partial\Omega$,  denoted by $d_\Omega$, is the function defined by
$$ 
d_\Omega(x)=\left\{\begin{array}{rl}
d(x,\partial\Omega)\hspace{1cm}\mbox{if }x\in\Omega\\
-d(x,\partial\Omega)\hspace{1cm}\mbox{if }x\notin\Omega
\end{array}
\right.
$$
where, as usual, $d(x,\partial\Omega)=\inf\limits_{y\in\partial\Omega}|x-y|$.
We write $d$ instead of $d_\Omega$ when there is no possible mistake for $\Omega$.
It is well-known that $d_\Omega$ is a   $1$-Lipschitz function which coincides with the unique viscosity solution of the eikonal equation
\begin{align*}
\begin{cases}
|Du|=1\hspace{2cm}x\in\Omega\\
u=0\hspace{2.57cm}x\in\partial\Omega\,.
\end{cases}
\end{align*}
Moreover, if we require some regularity for the set $\Omega$, we obtain further regularity for $d_\Omega$.\\
\begin{defn}
Let $K\subseteq\mathbb{R}^N$. We say that $K$ is a compact domain of class $\mathcal{C}^{2}$ if $K$ is a compact connected set and $\exists M\in \mathbb N$ such that $\forall 1\le i\le M$ $\exists B_{r_i}(x_i)$, $x_i\in\partial K$ and a function $\phi_i: B_{r_i}(x_i)\to\mathbb{R}$ such that
\begin{itemize}
\item[(i)] $\partial K\subseteq {\bigcup_{i=1}^M} \, B_{r_i}(x_i)$
\item[(ii)] $\partial K\cap B_{r_i}(x_i)=\{\phi_i=0\}$
\item[(iii)] $\phi_i$ is of class $\mathcal{C}^{2}$ with $D^2 \phi_i$ bounded in $B_{r_i}(x_i)$.
\end{itemize}
\end{defn}
In the following, we assume that  $\Omega$ is an open set such that $\overline{\Omega}$ is a compact domain of class $\mathcal{C}^{2}$. We set
\begin{align*}
&R_\eps:=\{x\in \R^N\,:\, |d_\Omega(x)|<\eps\}\\
&\Gamma_\eps:=\{x\in\Omega\,:\,  d_\Omega(x)<\eps\}=R_\eps\cap\Omega\,.
\end{align*}
We recall (see e.g. \cite{CDF} and \cite{cingul}) that  
\begin{align*}
\overline{\Omega}\quad \hbox{is a  compact domain of class }\mathcal{C}^{2}\iff\exists\eps_0>0\, : \quad d_\Omega\in\mathcal{C}^{2}(R_{\eps_0})
\end{align*}
and
\begin{equation}\begin{split}\label{omegadelta}
	\forall x\in\Gamma_{\eps_0}\ &\exists!\ \overline{x}\in\partial\Omega\mbox{ s.t. }d_\Omega(x)=|x-\overline{x}|\\
	&\mbox{ and }D d_\Omega(x)=D d_\Omega(\overline{x})=-\nu(\overline{x})
	\end{split}
	\end{equation}
	where $\nu$ stands for the outward unit normal to $\partial\Omega$.
\\\\

\textbf{Remark:} Actually, we will use the function $d_\Omega$ only near $\partial\Omega$, where this  is a regular function. So, from now on, when we will write $d_\Omega$ (or $d$ when there is no possible mistake) we will mean a $C^2(\overline{\Omega})$ function $\tilde{d}$ such that $\exists\eps_0>0$ with $\tilde{d}=d$ in $\Gamma_{\eps_0}$.
\\\\
For every  $r\in\mathbb{R}$ we set
\be\label{qn}
D_r=\left\{x\in\Omega\mbox{ s.t. }d(x)\ge\frac{1}{r}\right\}=\Omega\setminus\Gamma_{\frac 1 r}
\ee
so that we build a sequence $\{D_n\}_{n\in\mathbb{N}}$ of compact domains of class $\mathcal{C}^{2}$ such that
$$
D_n\subseteq\overset{\circ}{D}_{n+1}\qquad\mbox{and }\quad  \bigcup_{n=1}^\infty D_n=\Omega\,.
$$
For each couple of vectors $(v,w)\in\mathbb{R}^n\times\mathbb{R}^m$, the tensor product $v\otimes w$ denotes the $n\times m$-matrix $v^Tw$. Finally,  throughout the proofs we use the notation  $C$ to denote a generic constant which may vary from line to line.

\vskip0.4em
\subsection{Standing assumptions.} \quad Let us now make precise the assumptions on the coefficients $a_{ij}$ and on the nonlinearities $H,F,G$ of the system \rife{mfg}.
\vskip0.4em
We assume that $a(x)= (a_{ij}(x))_{ij}	$ is a  $N\times N$-matrix which belongs to $W^{1,\infty}(\Omega)^{N \times N}$ and satisfies
\begin{equation}\label{deg}
 a(x)\xi\cdot \xi  > 0\quad \forall\ x\in \Omega, \forall \xi\in  \mathbb{R}^N\,.
\end{equation}
We call $\lambda_r\in\mathbb{R}_{>0}$ the constant of uniform ellipticity of $a$ in $D_r$, i.e.
\begin{align}\label{ell}
a(x)\xi\cdot \xi \ge\lambda_r|\xi|^2\hspace{1cm}\forall x\in D_r\,,\quad \forall\xi\in\mathbb{R}^N\,.
\end{align}
Obviously we have $\lambda_r\le\lambda_s$ if $r\ge s$. Moreover, by continuity the matrix $a(x)$ will be nonnegative on $\overline \Omega$, but it is allowed to vanish at the boundary, in which case $\lambda_r\searrow0$.
\vskip0.5em
We assume that $H(t,x,p)$ is a  function such that  $(t,x)\mapsto H(t,x,p)$ is measurable for any given $p\in \R^N$ and $p\mapsto H(t,x,p)$ is of class $C^1$ for almost every $(t,x)\in Q_T$. We assume in addition that
\be\label{convex}
p\mapsto H(t,x,p) \quad \hbox{is convex.}
\ee
\indent Concerning the growth of the Hamiltonian, we work assuming that  it has at most quadratic growth with respect to $p$ and is locally bounded with respect to $x$, with $H(t,x,0)$ globally bounded. Precisely, we assume that 
\be\label{H0}
H(t,x,0) \in \parelle\infty
\ee
and
\be\label{linder}
\begin{split} & 
\hbox{$\forall$ compact set $K\subset \Omega$, \,\, $\exists C_K>0$: }
 \\
 &\qquad  | H_p(t,x,p)| \le C_K (1+|p|) \qquad \forall p\in \R^N ,\,\, \mbox{and a.e.} \,\, x\in K, t\in[0,T]\,.
\end{split}
\ee
Of course, \rife{H0}--\rife{linder} imply, by integration, that
\be\label{quadgrow}
| H(t,x,p)| \le C_K (1+|p|^2)\qquad \forall p\in \R^N ,\,\, \mbox{and a.e.} \,\, x\in K, t\in[0,T]\, 
\ee
for a  possibly different constant $C_K$.  
\\
The invariance condition will be formulated in terms of the diffusion matrix $a(x)$ and the Hamiltonian function $H(t,x,p)$. Namely, we assume that there exist $\delta>0$ and $C>0$ such that  the following inequality holds:
\begin{equation}\label{invariance}
\begin{split}
&\mathrm{tr}(a(x)D^2d(x))-H_p(t,x,p)Dd(x)   \ge \frac{a(x) Dd(x) \cdot Dd(x)}{d(x)}- C d(x)\,
\\
&  \qquad  \hbox{$\forall\ p\in\mathbb{R}^N$ and a.e. $x\in\Gamma_\delta$, $t\in[0,T]$,}
\end{split}
\end{equation}
where, we recall, $\Gamma_\delta$ is the subset of $\Omega$ with $d(x)<\de$.

A typical case when assumption \rife{invariance} is satisfied occurs if there exists a  $N\times N$-matrix $\sigma\in W^{1,\infty}$ such that $a=\sigma\sigma^*$, 
\be\label{zehro}
\sigma^*(x)Dd(x)=0\hspace{2cm}\forall x\in\partial\Omega\,,
\ee
and
$$
\mathrm{tr}(a(x)D^2d(x))-H_p(t,x,p)Dd(x)\ge 0
$$
for all $p\in\mathbb{R}^N$ and all $x$ in a neighborhood of $\partial \Omega$. Moreover, in the case that $H_p(t,x,p)$ is  Lipschitz with respect to $x$ (uniformly in $t$ and $p$), the inequality can be required to hold only for $x\in \partial \Omega$, since \rife{invariance} will be equally satisfied for $x$ in some $\Gamma_\de$ provided the constant $C$ is large enough. This is the typical condition which is given in the literature for linear operators, i.e. if $H(t,x,p)= b(x)\cdot p$, see e.g. \cite{CDF}. 

However, we stress that assumption \rife{invariance} is meant to include more general examples. On one hand, this condition includes the case of $a=\sigma\sigma^*$ with $\sigma$ being only $1/2$-H\"older continuous. On another hand, even the uniformly elliptic case is included in our setting, indeed $a(x)$ could be non degenerate at the boundary  provided the drift part is sufficiently coercive in the (inward) normal direction. Situations of this kind were considered, for instance, in \cite{LP}.

\vskip0.5em
Finally, the assumptions on the coupling costs $F,G$. Here we assume that  $F$ is a map from $Q_T \times C^0([0,T]; \elle1) $ into $\R$. In particular, for any given $m\in C^0([0,T]; \elle1)$, $F(\cdot,\cdot,m)$ defines a function on $Q_T$.  We assume that 
\be\label{F}
\begin{split}
& \hbox{$m\mapsto F(\cdot,\cdot, m) $ maps bounded sets of $C^0([0,T]; \elle1)$ into   bounded sets of $L^\infty(Q_T)$,}
\\
& \qquad \hbox{and is continuous in the  $L^1(Q_T)$- topology.}
\end{split}
\ee
We wish to include two model examples  in the previous conditions. The simplest case is when $F$ acts locally on the density 
$m(t,x)$: this means, for instance,  that $F$ is given through a real function $f:Q_T \times \R \to \R$ so that 
$$
F(t,x, m):= f(t,x,m(t,x))\,.
$$
In this case the condition is satisfied whenever $f$ is continuous with respect to $m$ and is uniformly bounded.
A second class of examples is given by nonlocal functions $F$, as, for instance, $F= K \star m $ for some bounded convolution kernel $K$. 

A similar condition is assumed for $G$, although here we need to strengthen the requirements in order to ensure that $Du$ is bounded up to $t=T$ (unless $H$ is Lipschitz continuous, see also Remak \ref{rem-couplings}). Namely, we assume that $G$ is a map from $\Omega \times  \elle1 $ into $\R$ such that 
\be\label{G}
\begin{split}
& \hbox{$m\mapsto G(\cdot, m) $ is a  continuous  map from $ \elle1$ into $ \elle1$}
\\
& \quad \hbox{which maps bounded sets of $\elle1$ into  bounded sets of  $W^{1,\infty}(\Omega)$.}
\end{split}
\ee
As is customary in mean field game systems, we will require in addition some monotonicity of $F,G$ in order to have uniqueness of solutions.
\vskip0.5em

\subsection{Short statement of the main results.}  

We list here the three main results that we prove in the paper, standing on the assumptions previously introduced. The first one is just concerned with the Hamilton-Jacobi-Bellman equation. The notion of weak solution is a standard one and will be precisely given in Definition \ref{defhjb}. Under the invariance condition, it turns out that  the problem is well-posed in the class of (globally) bounded solutions, with no need of prescription of the boundary condition.

\begin{thm}\label{teo1}
Assume that  $a(x)$ and  $H(t,x,p)$ satisfy assumptions  \rife{deg},   \rife{convex}-\rife{linder} and the invariance condition \rife{invariance}, and that  $G\in L^{\infty}(\Omega)$. 

Then there is one and only one bounded weak solution of the problem
$$
\begin{cases}
-\partial_t u - \sum\limits_{i,j} a_{ij}(x) \partial_{ij}^2 u +H(t,x,Du)=0\,, \qquad (t,x)\in (0,T)\times\Omega\\
u(T)=G(x )\,,\qquad x\in  \Omega\,.
\end{cases}
$$
\end{thm}

The second result gives, somehow, a counterpart for the Fokker-Planck equation. Indeed, under  the invariance condition the problem turns out to be well-posed in $\elle1$. Here the notion of weak solution is defined in a  dual way, see Definition \ref{deffp}, and incorporates somehow a transparent Neumann condition at the boundary.

With a slight abuse of notation, we denote by $L^\infty([0,T]; L^\infty_{loc}(\Omega))$ the space of measurable functions in $Q_T$ which are bounded on $(0,T)\times K$ for every compact subset $K\subset \Omega$.

\begin{thm}\label{teo2} 

Let $m_0\in  \elle1$, $m_0\ge0$.
Let $a\in W^{1,\infty}(\Omega)$ satisfy \rife{deg}.   
Assume that $  b\in L^\infty([0,T]; L^\infty_{loc}(\Omega))$ and that there exist  
$\delta_0, C>0$ such that the following inequality holds:
\begin{equation}\label{invariance-fp}
\mathrm{tr}(a(x)D^2d(x))-b(t,x)\cdot Dd(x)\ge   \frac{a(x) Dd(x) \cdot Dd(x)}{d(x)} - C\, d(x)
\end{equation}
for almost every $t\in (0,T)$ and  $x\in\Gamma_{\delta_0}$.  

Then there is one and only one  weak solution (in the sense of Definition \ref{deffp}) of the problem 
$$
\begin{cases}
\partial_t m - \sum\limits_{i,j} \partial_{ij}^2 (a_{ij}(x)m) -\mathrm{div}(m\, b(t,x))=0\,, \qquad (t,x)\in (0,T)\times\Omega\\
m(0)=m_0(x )\,,\qquad x\in  \Omega\,.
\end{cases}
$$
\end{thm}

Our third main result is concerned with the mean field game system, where we join the two previous results, using the conditions on the coupling terms and the viability assumption on the Hamiltonian. 

\begin{thm}\label{teo3} Assume that hypotheses \rife{deg}, \rife{convex}-\rife{linder}, \rife{invariance}, \rife{F} and \rife{G} hold true. 
Then there exists  one solution $(u,m)$ of  \rife{mfg}, in the sense of Definition  \ref{sol-mfg}.

If, in addition, $F$ and $G$ are monotone with respect to $m$, then the solution is unique.
\end{thm}

The further results that we prove are concerned with the regularity of solutions, but in this case we postpone the proper statements to Section 7.

\vskip0.5em

\subsection{Probabilistic interpretation and examples.}

Now we give the  probabilistic interpretation of the system. 
Given a probability space $(\tilde{\Omega},(\mathcal{F}_t)_t,\mathbb{P})$ (we use $\tilde{\Omega}$ instead of the classical $\Omega$ to avoid confusion with the state space $\Omega$ previously defined) and a Brownian motion ${(B_t)}_t$ adapted to the filtration ${(\mathcal{F}_t)}_t$, we consider for $s>t$ the solution $(X_s)_s$ of the following stochastic differential equation:
\begin{equation}\label{SDE}
\begin{cases}
dX_s=b(s,X_s,\alpha_s)ds+\sqrt{2}\sigma(X_s)dB_s\\
X_t=x
\end{cases}
\end{equation}
where $b$ and $\sqrt{2}\sigma$ are, as usual, the \textit{drift} and the \textit{diffusion} coefficients of the process $X_.$, and where the control $\alpha_s$ is a progressively measurable process adapted to the filtration $\mathcal{F}_t$ and taking values in  $A\subseteq\mathbb{R}^N$.
\\
We recall the \emph{Ito's formula}: if $\phi\in\mathcal{C}^{1,2}[0,T]\times\R^N)$, then we have
\begin{equation*}\begin{split}
d\phi(s,X_s)&=\left(\phi_t(s,X_s)+\mathrm{tr}(a(X_s)D^2\phi(s,X_s))+b(s,X_s,\alpha_s)\cdot\nabla\phi(s,X_s)\right)ds\ +\\&+\sqrt{2}(\nabla\phi(s,X_s))^*\sigma(X_s)dB_s\,.
\end{split}\end{equation*}
In many applications, it is required that the process $(X_t)_t$  remains in $\Omega$ for every $t\ge0$ and {\it for all available controls}. This leads to the terminology of {\it invariance condition} for assumption \rife{invariance}, which is justified in view of the following result. 

\begin{prop}
Let  $\sigma\in W^{1,\infty}(\Omega)$ and $b(s,x,\alpha)$ be (locally) Lipschitz with respect to the time and space variables, with a Lipschitz constant (locally) uniform in $\alpha$, and suppose that, for some $\de>0$ and $C>0$:
\begin{equation}\label{daprato}\begin{split}
& \mathrm{tr}(a(x)D^2d(x))+ b(s,x,\alpha)\cdot Dd(x) \ge \frac{a(x) Dd(x) \cdot Dd(x)}{d(x)}- C d(x)\,
\\ & \quad \forall (s,x)\in[t,T]\times\Gamma_\delta\,,\quad \forall \alpha \in A   \,,
\end{split}\end{equation}
where $a=\sigma\sigma^*$. Then, if $(X_s)_s$ is the solution of \eqref{SDE} with starting point $x\in\Omega$, we have
\begin{equation}\label{probinv}
\mathbb{P}(\{X_s\in\Omega\spazio\forall s>t\})=1\spazio.
\end{equation}
\begin{proof}
The proof is classical (see e.g. \cite{BCR}, \cite{CDF} for similar results) but we include it  for the reader's convenience and because condition \rife{daprato} applies to a more general setting than usual.\\
Let $(X_s)_s$ be the process solving \eqref{SDE}. The existence and uniqueness of  $X_t$
is ensured by the local Lipschitz character of $b$, $\sigma$.
For a bounded set $E\in\R^N$ we call $\tau_E$ the exit time from $E$ of the process $X_s$: for $\omega\in\tilde{\Omega}$
$$
\tau_E(\omega)=\inf\left\{ s\ge t\ |\ X_s(\omega)\notin E\right\}\,.
$$
So, proving \eqref{probinv} is equivalent to prove that $\P\left(\tau_\Omega<+\infty\right)=0$. To this purpose, we will show that, for all $s>t$, we have 
\begin{equation}\label{metroA}
\P\left(\tau_\Omega\le s\right)=0\,.
\end{equation}
Indeed, since
\begin{equation*}\begin{split}
&\left\{\tau_\Omega\le s \right\}\subseteq\left\{\tau_\Omega\le r \right\}\qquad\mbox{for }s\le r\\
&\mbox{and}\qquad \bigcup\limits_s\left\{\tau_\Omega\le s \right\}=\left\{\tau_\Omega<+\infty\right\}\,,
\end{split}\end{equation*}
then the assertion follows thanks to the monotone convergence theorem.  To prove \eqref{metroA}, we will show that 
\begin{equation*}
V(x):=-\log(d(x)) 
\end{equation*}
is, roughly speaking,  a super solution  up to a  constant.  Indeed, 
according to \eqref{daprato}  we obtain
\begin{align*}
\mathrm{tr}(a(x)D^2 V)+b(s,x,\alpha)\cdot DV & 
=-\frac{\mathrm{tr}(a(x)D^2d(x))+ b(s,x,\alpha)\cdot Dd(x)}{d(x)} \\
& \quad +\frac{a(x)Dd(x)\cdot Dd(x)}{d(x)^2}\le C\,,
\end{align*}
for each $(s,x)\in[t,T]\times\Gamma_{\delta}$ and for each $\alpha\in A$. 
Now, by a standard localization argument, we obtain a non-negative $\mathcal{C}^2$ function $U$ such that
\begin{equation*}
\begin{cases}
U(x)=V(x)&\mbox{for }x\in\Gamma_{\frac{\delta}{2}}\,,\\
\mathrm{tr}(a(x)D^2U(x))+b(s,x,\alpha)\cdot DU(x)\le C\qquad&\mbox{for }x\in\Omega\,,\ \,s\in[t,T]\,,\ \,\alpha\in A\,;
\end{cases}
\end{equation*}
we recall that the constant $C$ can change from line to line. 
To conclude, we consider a sequence of compact domains $\{D_n\}_n$ converging to $\Omega$, and the associated stopping times $\tau_{D_n}$. Applying Ito's Formula to $U$ and taking the expectation, we have
\begin{align*}
\E\left[U(X_{s\wedge\tau_{D_n}})\right]=U(x)+\E\left[\int_t^{s\wedge\tau_{D_n}}\left(\mathrm{tr}(a(X_r)D^2U(X_r)+b(r,X_r,\alpha_r)\cdot DU(X_r)\right)dr\right] 
\end{align*}
hence
$$
\E\left[U(X_{s\wedge\tau_{D_n}})\right]\leq U(x)+C(s-t)<+\infty 
$$
since $x\in\Omega$. Using Fatou's Lemma we get
\begin{align*}
\E\left[U(X_{s\wedge\tau_{\Omega}})\right]\le\liminf\limits_{n\to\infty}\E\left[U(X_{s\wedge\tau_{D_n}})\right]\le U(x)+C(s-t)<+\infty\,.
\end{align*}
Since $U(x)$ blows-up if $x\in \partial\Omega$, this implies
\begin{align*}
\P(\tau_\Omega\le s)=0\,.
\end{align*}
\end{proof}
\end{prop}

For $t\in[0,T]$ and $x\in\Omega$, we define now the value function  
\begin{align*}
u(t,x):=\inf\limits_{(\alpha_s)_s\subseteq A}\mathbb{E}\left[\int_t^T\left(F(s,X_s,m(s))+L(s,X_s,\alpha_s)\right)ds+G(X_T,m(T))\right]\spazio,
\end{align*}
where, for every $s$, $m(s)$ is a probability density function.\\ 
Here $F$ and $G$ are the \emph{cost functions} and the \emph{Lagrangian} $L$ satisfies standard conditions. Typically, we require the strict convexity of the function $L$. In a  usual way, one can apply the Ito's formula and the dynamic programming principle to obtain a Hamilton-Jacobi-Bellman equation for the function $u$. So defining the Hamiltonian $H$ as
\begin{align*}
H(t,x,p):=\sup\limits_{\alpha\in A}\left(-b(t,x,\alpha)\cdot p-L(t,x,\alpha)\right)\spazio,
\end{align*}
the value function $u$ turns out to be the solution of the following equation:
\begin{align*}
\begin{cases}
-\partial_t  u -\mathrm{tr}(a(x)D^2u)+H(t,x,Du)=F(t,x,m)\\
u(T)=G(x,m(T))\,.
\end{cases}.
\end{align*}
Moreover, we obtain a feedback optimal control $b=-H_p(s,x,Du(s,x))$. Plugging this control in \eqref{SDE} and solving the SDE gives the optimal trajectories $(\tilde{X}_s)_s$. We say that the system is in equilibrium if  the law of $X_s$ coincides with  $m(s)$ for every $s\in [0,T]$. Since the law of a solution of \eqref{SDE} solves a Fokker-Planck equation, we obtain a couple $(u,m)$ solution of \eqref{mfg}.

We give here two typical examples in which the hypotheses made so far upon $H$ are satisfied. In particular, the invariance condition may be satisfied by using controlled perturbations of linear invariant processes. 
\vskip1em
{\bf Example 1 }{\it (bounded controls)}. 

We consider a set of controls $A\subset\mathbb{R}^N$ which is compact and a positive number $M$. We take 
\begin{align*}
b(x,\alpha):=MDd(x)+\alpha\,.
\end{align*}
Then, for any $a(x)$ such that $a(x)Dd(x)\cdot Dd(x)=0$ on $\partial \Omega$,  the invariance condition  \eqref{invariance}  is satisfied for $M$ sufficiently large.
\\
Indeed, since the set of controls is bounded and  $a$ is Lipschitz, we have
$$
\frac{a(x) Dd(x)\cdot Dd(x) }{d(x)} - \mathrm{tr}(a(x)D^2d(x))- \alpha\cdot Dd(x) \leq C
$$
for some constant $C>0$ independent of $\alpha$. So \rife{invariance} holds provided $M$ is large enough. Let us now check the other conditions assumed upon $H$. 
In this situation, the Hamiltonian $H$ takes the form
\begin{align*}
H(x,p)=\sup\limits_{\alpha\in A}\left(-\alpha\cdot p-M Dd(x)\cdot p- L(x,\alpha)\right)\spazio.
\end{align*}	
Of course  $H$ is a convex function with respect to the variable $p$. Assume further that $L$ is strictly convex with respect to the last variable. Then the supremum that arises in the definition of $H(x,p)$ is attained at a unique point,  say $\alpha_{p,x}$, and the mapping $(x,p)\mapsto \alpha_{p,x}$ is continuous. In particular,
$H(x,p)$ is a continuous function and one can further  check  that it is differentiable with respect to $p$ with 
\be\label{hpex}
H_p(x,p)=-b(x,\alpha_{p,x}) = - M Dd(x) -\alpha_{p,x}\spazio,
\ee
which is continuous in both arguments and  bounded uniformly with respect to $(x,p)$. Thus, $H$ satisfies \rife{convex}--\rife{linder}.

In addition, we notice that,  if $L$ is Lipschitz with respect to $x$,  uniformly  in $\alpha$, then $H$ is Lipschitz with respect to $x$ and
$$
H_x= - M\, D^2 d(x) p - L_x (x,\alpha_{p,x})\,,
$$
which satisfies the growth condition $|H_x(x,p)|\le C(1+|p|)$.
If we have further that $L$ is $C^1$ with respect to $x$ with $L_x$ continuous with respect to $\alpha$, then $H$ is also $\mathcal{C}^1$ in both variables.

Finally, let us stress that  a similar example can be adapted to the case that $a(x)$ does not degenerate at the boundary, namely if $a(x) >0$ in $\overline \Omega$. In that case, it is enough to take $b(x,\alpha)= M\, \frac{Dd(x)}{d(x)} + \alpha$ in order to build a similar  example. 
%
%

\vskip1em
{\bf Example 2 }{\it (unbounded controls and coercive Hamiltonian)}. 

Here we consider  a case in which the set of controls is unbounded. This gives us an Hamiltonian with super-linear growth in $p$.

We take $A=\{\alpha\in\R^N\mbox{ s.t. }\alpha_i\ge0\ \forall i\}$ and we set
\begin{align*}
b(x,\alpha):=MDd(x)+B(x)\alpha\spazio,
\end{align*}
where $\forall x$ $B(x)$ is a $N\times N$-real valued matrix. Let $c_0\ge0$ be such that
	$$
	\mathrm{tr}(a(x)D^2d(x))\ge-c_0\spazio.
	$$
Then we have
\begin{align*}
\mathrm{tr}(a(x)D^2d(x))+b(x,\alpha)Dd(x)\ge M-c_0+ B(x)\alpha \cdot Dd(x) \ge M-c_0\spazio,
\end{align*}
if we choose $B$ such that
\be\label{Bx}
B(x)\alpha \cdot Dd(x)\ge0\spazio,\hspace{1cm}\forall\alpha\in A\spazio.
\ee
For example, we can take $B(x)_{ij}=Dd(x)_i \delta_{ij}$.  If \rife{Bx} holds, then the invariance condition \rife{invariance} is satisfied provided $M$ is sufficiently large.

Let us suppose further that the matrix $B(x)$  is bounded and continuous. As in the previous example, let the function $L(x,\alpha)$ be continuous in both arguments and strictly convex in $\alpha$, and assume the following coercivity condition:  $\exists \eta>0$, $q>1$, $c_0>0$ such that
$$
L(x,\alpha)\ge\eta|\alpha|^q-c_0\hspace{2cm}\forall\alpha\in A\spazio,\spazio x\in\overline{\Omega}\spazio.
$$
Then one can readily check the properties of $H$ as before. In particular, the supremum that arises in the definition of $H(x,p)$ is attained at a unique point $\alpha_{p,x}$, which is continuous with respect to $(x,p)$ and now satisfies the estimate:
$$
|\alpha_{p,x}| \leq C (1+ |p|^{q'-1}) \qquad \forall (x,p)\,,
$$
where $q'$ is the conjugate exponent of $q$, i.e. $q'=\frac{q}{q-1}$. As a consequence, there exists a constant $C>0$ such that
$$
-C(1+|p|)\le H(x,p)\le C(1+|p|^{q'})\spazio.
$$
Similarly, the  differentiability of $H$ with respect to $p$ and formula \rife{hpex} imply that
$$
|H_p(x,p)| \leq C (1+ |p|^{q'-1})\,,
$$
so that  $H$ has at most quadratic growth  for $q\geq 2$ and satisfies \rife{convex}--\rife{linder}.

Moreover, if $L$ and $B$ are $\mathcal{C}^1$ with respect to $x$, and the derivative of $L$ is continuous in $\alpha$, then $H$ is $\mathcal{C}^1$ in both variables. Finally, we notice that, if $L_x$ has a linear growth in $\alpha$, then $|H_x(x,p)|\le C(1+|p|^{q'})$.

\section{The Hamilton-Jacobi-Bellman equation}

In this Section we study the  Hamilton-Jacobi-Bellman (HJB) equation 
\begin{equation}\label{hjb}
\begin{cases}
-u_t-\mathrm{tr}(a(x)D^2u)+H(t,x,Du)=0 & \hbox{in $(0,T)\times \Omega$,}\\
u(T)=G(x) & \hbox{in $ \Omega$}
\end{cases}\spazio.
\end{equation}
under conditions of invariance of the domain (see \rife{invariance}).  In particular, we observe that no boundary condition is prescribed. Here, the boundedness of $u$ will be enough to characterize the solution.

Note that in this formulation we do not have a function $F$, since, for $m$ fixed, it can be included in the Hamiltonian $H$. Since $F$ does not depend on $Du$, the derivative $H_p$ does not change, and the conditions of invariance \eqref{invariance} are not affected from this inclusion.

We assume that $G$ is a bounded function  and the Hamiltonian $H$ satisfies \rife{H0} and  \rife{quadgrow}.

Different notions of solutions could be used for problem \rife{hjb}. Since we are assuming $a$ to be Lipschitz continuous, the problem can be formulated in divergence form as well, namely
\begin{equation*}
\begin{cases}
-u_t-\mathrm{div}(a(x)Du)+\tilde{b}(x)\cdot Du+H(t,x,Du)=0 \\
u(T)=G(x)
\end{cases}
\end{equation*}
where $\tilde{b}$ is defined as
\begin{equation}\label{tilde}
\tilde{b}_j(x)=\overset{N}{\underset{i=1}{\sum}}\frac{\partial a_{ij}}{\partial x_i}(x)\spazio,\hspace{2cm}j=1,\ldots N\spazio.
\end{equation}
This fact allows us  to use a weak (distributional) formulation which avoids any continuity requirement on the solution as well as on $F,H$ with respect to $t$ and $x$. The natural growth condition \rife{quadgrow} also leads us to consider local $H^1$ solutions as defined below.
\vskip0.4em

\begin{defn}\label{defhjb}
We say that $u$ is a weak solution (resp. subsolution, supersolution) of the problem \eqref{hjb} if
\begin{itemize}
			\item[(i)] $u\in L^\infty([0,T]\times\Omega)$;
			\item[(ii)] $u\in L^2([0,T];W^{1,2}(K))$ for each $K\subset\subset\Omega$;
			\item[(iii)] $\forall\phi\in C_c^\infty((0,T]\times\Omega)$ (resp. $\ge0$) the weak formulation holds:
			\begin{equation*}\begin{split}
			\intif &u\phi_t\spazio dxdt\ + \intif a(x)Du\cdot D\phi\, dxdt\ +\\+ &\intif (H(t,x,Du)+\tilde{b}(x)Du)\phi\spazio dxdt=\into G(x)\phi(T)dx\spazio,
			\end{split}\end{equation*}
			(resp. $\le$, $\ge$), where $\tilde{b}$ is defined above.
		\end{itemize}
		\end{defn}
		
\vskip0.4em
\begin{rem} 
We observe that, if $u$ is a weak solution of \eqref{hjb}, then $u\in\mathcal{C}([0,T];L^p(\Omega))$ for each $p\ge1$.\\
Indeed, from $(ii)-(iii)$ we have $u_t\in L^2([0,T];W^{-1,2}(K))+L^1((0,T)\times K)$ for each $K\subset\subset\Omega$, and so $u\in C([0,T]; L^1(K))$ (see \cite[Theorem 2.2]{P-Arma}). 
Since $u\in L^\infty([0,T]\times\Omega)$, one can actually conclude that $u\in C([0,T];L^p(\Omega))$ for each $p\ge1$.
\end{rem}

\subsection{Existence of solutions}

We start by proving the existence of at least one weak solution. This is achieved without using the invariance condition and actually follows by a  standard use of global bounds and local compactness. The next lemma is by now standard, following the arguments in \cite{BMP}. For the reader's convenience, since the statement may not be found exactly in this form  in previous references, we will give a short proof in the Appendix. 

\begin{lem}\label{compact-hjb}  Let $\{\Omega_{\eps}\}$ be a sequence of domains such that $\Omega_{\eps}\subseteq\Omega_\eta\subseteq\Omega$ for $\vep>\eta$, and $\bigcup_\vep \Omega_\vep=\Omega$.  Let $H_\vep$ be a sequence of Carath\'eodory functions such that 
$$
|H_\vep(t,x,p)| \leq C_\vep (1+ |p|^2) \quad  \hbox{for a.e. $(t,x)\in  (0,T)\times \Omega_\vep$ and $\forall p\in \R^N$,}
$$
where $C_\vep$ is  bounded (independently of $\vep$) in $(0,T)\times K$, for every compact set $K\subset \Omega$.

Assume that $a_\vep$ is a sequence of matrices which is uniformly bounded and, locally, uniformly coercive.
Let $u_\vep$ be a sequence of solutions of 
\be\label{hjbeps}
-(u_\eps)_t-\mathrm{div}(a_\vep (x)Du_\eps)+H_\vep(t,x,Du_\eps)=0\,,\quad (t,x)\in[0,T]\times\Omega_{\eps}\,,
\ee
such that  $\|u_\vep\|_\infty$ is uniformly bounded.

Then there exists a function $u\in L^\infty(Q_T) \cap L^2(0,T; W^{1,2}_{loc}(\Omega))$ and   a subsequence $u_\vep$ converging to $u$ weakly in $L^2(0,T; W^{1,2}(K))$ and strongly in $L^p((0,T)\times K)$, for all compact sets $K\subset\subset\Omega$ and all $p<\infty$.

In addition, if $a_\vep(x)$ converges almost everywhere in $\Omega$ to some matrix $a(x)$, then  $u_\vep$ converges to $u$ strongly in $L^2(0,t; W^{1,2}(K))$
for all $t<T$, and the convergence holds up to $t=T$  if $u_\vep(T)$ converges almost everywhere in $\Omega$.
\end{lem}

\begin{proof} See the Appendix. 
\end{proof}

From the above Lemma we deduce the following stability result.
 
\begin{prop}\label{stab-hjb}
Let $\{\Omega_{\eps}\}$ be a sequence of domains such that $\Omega_{\eps}\subseteq\Omega_\eta\subseteq\Omega$ for $\vep>\eta$, and $\bigcup_\vep \Omega_\vep=\Omega$.  Assume that $H_\vep(t,x,p)$ is a sequence of Carath\'eodory functions such that:
\begin{equation}\label{Heps}
\begin{split}& 
H_\vep(t,x,p) \to H(t,x,p) \qquad \hbox{for a.e. $(t,x)\in (0,T)\times\Omega $ and every $p\in \R^N$.} 
\\
&  
\|H_\vep(t,x,0)\|_\infty \leq C\,, 
\\
& |H_\vep(t,x,p)| \leq C_\vep (1+ |p|^2)  \qquad \hbox{for a.e. $(t,x)\in (0,T)\times\Omega_\vep$ and every $p\in \R^N$,}
\end{split}
\end{equation}
where $C$ is a constant independent of $\vep$ and $C_\vep$ is a constant which is bounded (independently of $\vep$) in $(0,T)\times K$, for every compact set $K\subset \Omega$.\\
Assume that $a_\eps$   is a sequence of matrices which is uniformly bounded and, locally, uniformly coercive and, moreover,  there exists a  matrix $a(x)$ such that
$$
a_\eps(x)\to a(x)\qquad\hbox{for a.e. $x\in\Omega$}\,.
$$ 
Finally, assume that $\{G_\eps\}_\eps$ is a sequence of uniformly bounded functions such that $\exists G\in L^\infty(\Omega)$ with
$$
G_\eps(x)\to G(x)\qquad\hbox{ for a.e. $x\in\Omega$}\,.
$$
Let  $u_\eps\in L^2([0,T];W^{1,2}(\Omega_\eps))$ be the unique solution of the  approximating system
\begin{equation}
\label{heps}
\begin{cases}
-(u_\eps)_t-\mathrm{div}(a_\eps(x)Du_\eps)+H_\vep(t,x,Du_\eps)=0 \qquad (t,x)\in (0,T)\times \Omega_\vep\\
u_\eps(T)=G_\eps(x)\\
a_\eps(x)Du_\eps\cdot\nu_{|\partial\Omega_\eps}=0\,.
\end{cases} 
\end{equation}
Then we have that there exists $u\in L^\infty(Q_T)\cap L^2(0,T; W^{1,2}_{loc}(\Omega))$ and a  subsequence $u_\vep$ such that 
$$
u_\vep \to u\qquad \hbox{in $L^2(0,T; W^{1,2}(K))\cap C^0([0,T]; L^p(\Omega))$ for any $p<\infty$}
$$
for all compact sets $K\subset \Omega$, and $u$ is a weak solution of 
$$
\begin{cases}
-u_t-\mathrm{div}(a(x)Du)+ H(t,x,Du)=0\,,\qquad (t,x)\in (0,T)\times \Omega \\
u(T)=G(x)\,.
\end{cases}
$$
\end{prop}

\begin{proof} Since $\|H_\vep(t,x,0)\|_\infty$ is uniformly bounded, by maximum principle we have that $\|u_\vep\|_\infty$ is uniformly bounded. From Lemma \ref{compact-hjb}, we deduce that there exists a function $u\in L^\infty(Q_T) \cap L^2(0,T; W^{1,2}_{loc}(\Omega))$ and   a subsequence $u_\vep$ converging to $u$ strongly in $L^p(Q_T)\ \forall p$ and  in $L^2(0,T; W^{1,2}(K))$, for all compact sets $K\subset \Omega$. In particular, we have that $Du_\vep\to Du$ in $L^2([0,T]\times K)$ for any compact subset $K$.
Thanks to the pointwise convergence of $H_\vep$ towards $H$, and due to the growth assumptions, we infer by Lebesgue theorem that 
$$
H_\vep(t,x,Du_\eps)\to H(t,x,Du)\hspace{2cm}\mbox{in }L^1([0,T]\times K) 
$$
for all compact sets $K\subset \Omega$. Since the equation implies that $(u_\eps)_t$ strongly converges in $L^2(0,T; W^{-1,2}(K))+ L^1([0,T]\times K)$, from the embedding result in \cite[Theorem 2.2]{P-Arma}, we deduce  that 
$$
u_\vep \to u\qquad \hbox{in $C^0([0,T]; L^1(K))$}
$$
and due to the $L^\infty$ bound the convergence actually holds in $C^0([0,T]; \elle p)$ for every $p<\infty$.
Now we can pass to the limit in the weak formulation of \eqref{heps}  to show that $u$ is  a weak solution.
\end{proof}

We finally deduce the existence of a solution for our problem.

\begin{thm}\label{exhjb} Suppose $G\in L^\infty(\Omega)$. Assume that $a$ satisfies \rife{deg} and that $H(t,x,p)$ satisfies \rife{H0} and \rife{quadgrow}. Then the problem \rife{hjb} has at least one solution.
\end{thm}
 	
\begin{proof}
We first notice that, replacing  $H(t,x,p)$ by $\tilde{H}(t,x,p)=H(t,x,p)+\tilde{b}(x)\cdot p$, where $\tilde b$ is defined in \rife{tilde}, the function $\tilde{H}$ satisfies the same hypotheses as $H$. So, without loss of generality,  it is enough to prove  the existence of solutions for the following equation:
\begin{equation}\label{div}
\begin{cases}
-u_t-\mathrm{div}(a(x)Du)+H(t,x,Du)=0\\
u(T)=G(x)\,.
\end{cases}
\end{equation}
Here we define the truncation of $H$ at levels $\pm \frac1\vep$:
$$
H_\vep(t,x,p):= \min\left\{ \max\left\{H(t,x,p),-\frac 1\eps\right\}, \frac1\vep\right\}\,.
$$
and we consider  $u_\eps\in L^2([0,T];W^{1,2}(\Omega))$ as the unique solution of the penalized Neumann problem
\begin{equation}
\label{hjbeps}
\begin{cases}
-(u_\eps)_t-\mathrm{div}((a(x)+\eps I)Du_\eps)+H_\vep(t,x,Du_\eps)=0\,,\qquad (t,x)\in (0,T)\times \Omega\,,\\
u_\eps(T)=G(x)\\
Du_\eps\cdot\nu_{|\partial\Omega}=0\,.
\end{cases}
\end{equation}
Applying  Proposition \ref{stab-hjb} with $\Omega_\vep=\Omega$, we conclude.
\end{proof}

We stress that an alternative way to construct a  solution of \rife{hjb} would be to use Neumann problems on  a sequence of domains converging to $\Omega$; the local ellipticity of $a(x)$ and the local boundedness of $H$ would avoid to approximate the nonlinearities, in this case. The proof is again a straightforward consequence of Lemma \ref{compact-hjb} and Propositon \ref{stab-hjb}.

\begin{prop}\label{altra-app-hjb} Let us set
$ \Omega^\vep:= \{ x\in \Omega\,:\, d(x) >\vep \}$ and let $u_\vep$ be the unique solution of the Neumann problem
\be\label{neumeps}
\begin{cases}
-(u_\eps)_t-\mathrm{tr}(a(x)D^2u_\eps)+H(t,x,Du_\eps)=0 & \hbox{in $(0,T)\times \Omega^\vep$,}\\
u_\eps(T)=G(x) & \hbox{in $\Omega^\vep$,}
\\
a(x)Du_\eps\cdot\nu_{|\partial\Omega^\vep}=0\,.
\end{cases}
\ee
Then, up to subsequences, $u^\vep$ converges (in $L^2(0,T; W^{1,2}(K))$, for all compact sets $K\subset \Omega$) to  a weak solution $u$ of problem \rife{hjb}.
\end{prop}

\begin{oss} \rm
The above compactness, and so the existence results, do not need that the matrix $a(x)$ be globally Lipschitz in $\Omega$, since a local Lipschitz condition is enough. 

Moreover, we point out that similar results could be proved for other kind of equations, including for instance fully nonlinear equations (Bellman operators, etc...). In fact, it is clear that two only ingredients  are required in the above construction: a  global $L^\infty$ bound (typically ensured by the  bound on $\|H(t,x,0)\|_\infty$) and a  local compactness and stability for equi-bounded solutions. For instance, in the fully nonlinear case this may  be achieved in the topology of uniform convergence in order to build a  viscosity solution inside. 
\end{oss}

		\subsection{Uniqueness of solutions}
Now we  prove that  the HJB equation \eqref{hjb} has a unique solution if the invariance condition holds. The strategy is classical and relies on the existence of a blow-up supersolution and the convexity of $H$; a  similar principle can be found e.g. in \cite[Lemma 6]{LP}. 

\begin{thm}\label{unichjb}
Suppose $G\in L^{\infty}(\Omega)$. Assume that $a(x)$ satisfies \rife{deg},  that $H(t,x,p)$ satisfies \rife{convex}-\rife{linder} and that the invariance condition \rife{invariance} holds true. 

Then there is at most one bounded weak solution of the problem \eqref{hjb}.
\end{thm}


\begin{proof}

Let $u$, $v$ be two bounded solutions  of \eqref{hjb}. For $\vep>0$, we set
$$
v_\vep= v+\eps^2 (M-\log d(x))+\eps\, \sqrt{T-t}\,.
$$
A straightforward computation implies
\begin{align*}
& -(v_\eps)_t-\mathrm{tr}(a(x)D^2v_\eps)+ H(t,x,Dv_\eps) =
-v_t-\mathrm{tr} (a(x)D^2v)+ H(t,x,Dv_\eps) \\
& \qquad +
\eps^2\left( \frac{\mathrm{tr}(a(x)D^2d(x)))}{d(x)}-\frac{a(x)Dd(x)\cdot Dd(x)}{d(x)^2}\right)+\frac{\eps}{2\sqrt{T-t}}\spazio.
\end{align*}
By convexity of $H$, we have
\begin{align*}
H(t,x,Dv_\eps) & \geq H(t,x,Dv) + H_p(t,x,Dv) \cdot (Dv_\vep-Dv)
\\
& = H(t,x,Dv) -\vep^2\, \frac{H_p(t,x,Dv)\cdot Dd(x)}{d(x)} 
\end{align*}
so we deduce
\be\label{veps}
\begin{split}
& -(v_\eps)_t-\mathrm{tr} (a(x)D^2v_\eps)+ H(t,x,Dv_\eps) \geq  
-v_t-\mathrm{tr} (a(x)D^2v)+ H(t,x,Dv) \\
& +\eps^2\left( \frac{\mathrm{tr}(a(x)D^2d(x)))- H_p(t,x,Dv)\cdot Dd(x)}{d(x)}-\frac{a(x)Dd(x)\cdot Dd(x)}{d(x)^2}\right)+\frac{\eps}{2\sqrt{T-t}}\,.
\end{split}
\ee
Using assumption \rife{invariance}, we have that there exists $\de>0$ such that 
\be\label{use-inv}
\frac{\mathrm{tr}(a(x)D^2d(x)))- H_p(t,x,Dv)\cdot Dd(x)}{d(x)}-\frac{ a(x)Dd(x)\cdot Dd(x)}{d(x)^2}\geq - C
\ee
for all  $x\in \Gamma_\de$. However, due  to parabolic regularity (see e.g. \emph{Theorem V.3.1} of \cite{LSU}), we know that each solution of \eqref{hjb} is locally Lipschitz in the space variable, and in particular, thanks to the global $L^\infty$ bound of the solutions, we have
$$
|Dv(t,x)| \leq \frac {C_\de}{\sqrt{T-t}}\qquad \forall(t,x): \, t\in (0,T), \, d(x) \geq \de\,.
$$
Because of \rife{linder}, we deduce that 
$$
|H_p(t,x,Dv)| \leq C_\de\left( 1+ \frac 1{\sqrt{T-t}}\right)\forall(t,x): \, t\in (0,T), \, d(x) \geq \de\,.
$$
Therefore,  since $d(x)$ is a smooth extension of the distance function, the inequality \rife{use-inv}  extends  to the whole domain as follows:
\be\label{use-inv2}
\frac{\mathrm{tr}(a(x)D^2d(x)))- H_p(t,x,Dv)\cdot Dd(x)}{d(x)}-\frac{ a(x)Dd(x)\cdot Dd(x)}{d(x)^2}\geq - \frac C{\sqrt{T-t}} 
\ee
for any $(t,x)\in Q_T$, for a  possibly different constant $C$. Finally, using \rife{use-inv2} and the fact that $v$ is a super solution, we deduce from \rife{veps}
$$
-(v_\eps)_t-\mathrm{tr} (a(x)D^2v_\eps)+ H(t,x,Dv_\eps) \geq -\frac{C\, \eps^2}{\sqrt{T-t}}+ \frac \eps{2\sqrt{T-t}} \geq 0
$$
provided $\vep$ is  sufficiently small.

Thus, $v_\vep$ is a super solution and clearly $u-v_\vep<0$ in a  neighborhood of $\partial \Omega$ since $u,v$ are bounded while $\log d(x) \to -\infty$. For a convenient choice of $M$, we also have that $v_\vep(T)\geq v(T) \geq u(T) $. Therefore, $u$ and $v_\vep$ are a pair of sub and super solution in any  subset $ \Omega\setminus \Gamma_\eta$ and $u\leq v_\vep $ on $\partial (\Omega\setminus \Gamma_\eta)$ if $\eta$ is sufficiently small. 
We can apply e.g. Proposition 2.1 in \cite{PZ}\footnote{even if \cite[Proposition 2.1]{PZ} is written with $a(x)=I$, the same proof applies without any change to the case of  a bounded coercive matrix $a(x)$.} to conclude that $u\leq v_\vep$ in $ \Omega\setminus \Gamma_\eta$. Letting $\eta\to 0$, we get $u\leq v_\vep$ in $(0,T)\times \Omega$. As $\vep \to 0$, we obtain $u\leq v$. Reversing the roles of $u,v$, we conclude with the uniqueness of solutions.
\end{proof}
				
		
The above uniqueness result yields important 	consequences in terms of stability of solutions. Namely, under the invariance condition all different approximations, as those suggested in the previous subsection, converge towards the same solution.

 \begin{cor}\label{stability}
Let the assumptions of Theorem \ref{unichjb} hold true. Then, given any sequences  $a_\vep, H_\vep, G_\vep$ satisfying the hypotheses of Proposition \ref{stab-hjb}, the (whole) sequence  $u_\vep$ of solutions of  \rife{heps} converges to the unique weak solution $u$ of \rife{hjb}. In addition, $u$ is also the limit of the (whole) sequence of solutions of problem \rife{neumeps}.
 \end{cor}

\section{The Fokker-Planck equation}
In this section we turn the attention to the Fokker-Planck equation under the invariance conditions.
So we consider the following equation
 \begin{equation}\label{kfp}
\begin{cases}
m_t -   \sum\limits_{i,j} \partial_{ij}^2 (a_{ij}(x)m)-\mathrm{div}(m\, b(t,x))=0 & \hbox{in $Q_T$,}\\
m(0)=m_0  & \hbox{in $\Omega$,}
\end{cases}
\end{equation}
where  $b:[0,T]\times\overline{\Omega}\to\mathbb{R}^N$ is a vector field which is locally bounded in $[0,T]\times \Omega$ (i.e. it is bounded in $(0,T)\times K$, for any compact set $K\subset \Omega$).

\begin{defn}\label{deffp}
Let $m\in L^1([0,T]\times\Omega)$. We say that $m$ is a weak solution of \eqref{kfp} if
\begin{itemize}
\item [(i)] $m\in C([0,T] ;L^1(\Omega))$, $m\ge0$;
\item [(ii)] For each $\phi\in C([0,T];L^1(\Omega))\cap L^\infty([0,T]\times\Omega)$ such that $\phi$ satisfies
$$
\begin{cases}
-\phi_t-\mathrm{tr}(a(x)D^2\phi)+b\cdot D\phi\in L^\infty([0,T]\times\Omega)\\
\phi(T)=0
\end{cases}
$$ 
in the sense of Definition \ref{defhjb}, the weak formulation holds:
\begin{equation*}
\intif m\left(-\phi_t-\mathrm{tr}(a(x)D^2\phi))+b\cdot D\phi\right)dxdt=\into m_0\phi(0)dx\,.
\end{equation*}
\end{itemize}
\end{defn}

Since $a$ is assumed to be Lipschitz,  here we also have
$$
 \sum\limits_{i,j} \partial_{ij}^2 (a_{ij}(x)m)=\mathrm{div}(a^*(x)\, Dm)+\mathrm{div}(m\tilde{b}(x))\spazio,
$$
where $\tilde{b}(x)$ is defined by \rife{tilde}. So, there is no loss of generality in  considering only divergence form operators:
\begin{equation}\label{fp}
\begin{cases}
m_t-\mathrm{div}(a^*(x)Dm)-\mathrm{div}(m\,b(t,x))=0\\
m(0)=m_0
\end{cases}
\end{equation}
Weak solutions of \rife{fp} are defined by duality exactly as in Definition \ref{deffp}. In other words, $m$ is a weak solution of \rife{kfp} if and only if it is a weak solution of \rife{fp} with $b$ replaced by $b-\tilde b$. Hereafter, we will deal with problem \rife{fp}, where the adjoint matrix $a^*$ appears in the divergence operator, in order to keep consistency with the dual HJB equation considered before. 

\subsection{Existence of solutions}

As for the existence of solutions of Hamilton-Jacobi equation, we reason through approximation and use compactness arguments. We start with the following lemma, whose proof is postponed to the Appendix.

\begin{lem}\label{compact-fp} Let $\{\Omega_{\eps}\}$ be a sequence of domains such that $\Omega_{\eps}\subseteq\Omega_\eta\subseteq\Omega$ for $\vep>\eta$, and $\bigcup_\vep \Omega_\vep=\Omega$.  Let $b_\vep\in L^\infty(0,T; L^\infty_{loc}(\Omega))$ be a  sequence such that, for every compact set $K\subset \Omega$,  $b_\vep$ is bounded in $(0,T)\times K$ uniformly in $\eps$. 
Assume that $a_\vep$ is a sequence of matrices which is uniformly bounded and, locally, uniformly coercive.
Let  $m_\vep$ be a  solution  of
\be\label{meps}
(m_\eps)_t-\mathrm{div}(a_\vep^* (x)Dm_\eps+ b_\vep\, m_\vep)=0\quad\mbox{in }(0,T)\times\Omega_\eps\spazio,
\ee
such  that $\|m_\vep(t)\|_{\elle1}$ is uniformly bounded with respect to $\vep$ and $t\in [0,T]$.

Then there exists a function $m\in L^\infty((0,T);\elle1)$ 
and   a subsequence $m_\vep$ such that
$$
m_\vep \to m \qquad \hbox{in $L^1((0,T)\times K)$}
$$
for all compact sets $K\subset \Omega$. Moreover, if, for some $m_0$, $m_\eps(0)\to m_0$  in $L^1_{loc}(\Omega)$ and, for some $a(x)$, $b(t,x)$, we have $a_\eps(x)\to a(x)$ and $b_\eps(t,x)\to b(t,x)$ almost everywhere in $Q_T$,  then we also have  
$$
m_\eps\to m \qquad\mbox{in }C^0([0,T]; L^1(K))
$$
for all compact sets $K\subset \Omega$.
\end{lem}

\begin{proof}  See the Appendix.
\end{proof}

\vskip1em

The next step says that, under the invariance condition, there is a {\it global } $L^1$ stability.

\begin{prop}\label{stab-fp} 
Let $\{\Omega_{\eps}\}$ be a sequence of domains such that $\Omega_{\eps}\subseteq\Omega_\eta\subseteq\Omega$ for $\vep>\eta$, and $\bigcup_\vep \Omega_\vep=\Omega$.  Let $m_0\in \elle1$, $m_0\ge0$.
Assume that $b_\vep \in L^\infty([0,T]\times{\Omega_\vep})$ is such that
\be\label{beps}
\begin{cases}
  \hbox{$\forall $ compact $K\subset \Omega$, $b_\vep$ is uniformly bounded in $(0,T)\times K$,}
& \\
  b_\vep(t,x) \to b(t,x) \quad \hbox{a.e. in $Q_T$.} & 
\end{cases}
\ee
Let  $a_\eps(x)$ be a sequence of locally uniformly coercive and Lipschitz matrices such that $a_\vep(x)\to a(x)$ almost everywhere in $\Omega$. 

We call $d_\eps(x)=d_{\Omega_{\eps}}(x)$ and we assume that $b_\vep$ satisfies the following condition in $\Omega_{\eps}$: there exist  
$\delta_0, C>0$ and two sequences $r_\eps, h_\vep\to 0$ such that 
\begin{equation}\label{divereps}
\mathrm{div}(a_\eps(x)Dd_\eps(x))-b_\vep (t,x)\cdot Dd_\eps(x)\ge   \frac{a_\eps(x) Dd_\eps(x) \cdot Dd_\eps(x)}{d_\eps(x)+r_\eps} - C(d_\eps(x)+h_\eps)
\end{equation}
for all $\vep>0$ and a.e. $t\in (0,T)$, $x\in\Gamma_{\delta_0}\cap\Omega_{\eps}$.  

Let  $m_\vep$ be the solution, in $\Omega_{\eps}$, of the Neumann problem
\begin{equation}
\label{fpeps}\begin{cases}
(m_\eps)_t-\mathrm{div}(a_\eps^*(x)Dm_\eps)-\mathrm{div}(m_\eps\, b_\eps)=0 \qquad (t,x)\in (0,T)\times \Omega_\vep\,,\\
m_\eps(0)=m_0\\
(a_\eps^*(x) Dm_\eps + b_\eps\,m_\eps)\cdot\nu_{|\partial\Omega_\eps}=0\,.
\end{cases}\spazio.
\end{equation}
Then there exists $m\in \elle1$ such that (defining $m_\eps=0$ in $\Omega\setminus\Omega_{\eps}$) we have, up to a  subsequence,
$$
m_\vep \to m \qquad \hbox{in $C^0([0,T]; \elle1)$} 
$$
and $m$ is a weak solution of problem \rife{fp}.
\end{prop}

\vskip1em

\begin{rem} Despite the above statement is given in a general version, the reader should keep in mind at least two typical examples of approximations. The first one occurs if $a(x)$ degenerates on the boundary in the normal direction, i.e. if $a(x)Dd(x)\cdot Dd(x)=0$ on $\partial \Omega$, and if $b(t,x)$ is bounded in $Q_T$ and the  invariance property \rife{inv-lin} holds true. In this case we can use this result with $\Omega_\vep=\Omega$,  $b_\vep=b$ and $a_\vep(x)= a(x)+ \vep\, I$. Then \rife{divereps} is satisfied provided  $\vep= o(r_\vep)$ and with $h_\vep \simeq  \frac{\vep}{r_\vep}$. 
A second example occurs if the drift $b$ is unbounded near the boundary, which is certainly the case whenever $a(x)$ does not degenerate and the invariance condition \rife{inv-lin} holds. In this case one needs to work on internal domains and the above result can be used with $\Omega_\vep=\{x\,:\, d(x)>\vep\}$, $b_\vep=b$, $a_\vep= a$ and $r_\vep=h_\vep= \vep$ (see Theorem \ref{exfp} below). 
\end{rem}
\vskip0.4em
\begin{proof}
For simplicity, we divide the proof into steps. 
\vskip1em
\emph{Step 1: local convergence.} 
Hereafter, we extend $m_\vep$ to $\Omega$ defining $m_\eps=0$ in $\Omega\setminus\Omega_{\eps}$. Integrating the equation in \rife{fpeps}, one has immediately, for each $t\in[0,T]$, 
\begin{equation}
\label{prob}
\into m_\eps(t)dx=\int_{\Omega_{\eps}}m_0\,dx\to  \into m_0\, dx\,.
\end{equation}
Moreover, by the maximum principle, we have $m_\eps\ge0$ in $Q_T$.

We use Lemma \ref{compact-fp} to deduce that $m_\vep$ is relatively compact and, up to  a subsequence,  converges to some $m\in \elle1$; the convergence holds almost everywhere in $Q_T$ and, in addition, in $C^0([0,T]; L^1(K)$ for every compact subset $K\subset \Omega$.    

\vskip0.5em
\emph{Step 2: global $L^1$ convergence.} Now we want to prove that, for every $t\in [0,T]$,
$$
m_\eps(t)\to m(t)\hspace{2cm}\mbox{ strongly in }L^1(\Omega)\spazio.
$$
Since $m_\eps\ge 0$, it suffices to prove that, for every $t\in [0,T]$,
\be\label{ae+int}
m_\eps(t)\to m(t)\hspace{0.5cm} \mbox{ a.e. in $\Omega$}\hspace{1cm}\mbox{and}\hspace{1cm}\into m_\eps(t)dx\to\into m(t)dx\spazio.	
\ee
The almost everywhere  convergence of $m_\eps(t)$ to $m(t)$ is already given by Lemma \ref{compact-fp} (up to a subsequence, which is not relabeled).
For the convergence of the integrals, by \eqref{prob} we have to prove that
\begin{equation}\label{aim}
\into m(t)dx=\into m_0\, dx\spazio.
\end{equation}
By Fatou's lemma  and \rife{prob} we have 
\begin{equation}\label{lessie}
\into m(t)dx\le \into m_0\, dx\spazio,
\end{equation} 
which in particular implies that $m\in \elle1$. To prove the reverse inequality, for $\de>0$ we consider the the auxiliary function
\be\label{log}
\phi_\eps=\log(d_\eps(x)+\delta) -\log \de\,.
\ee
Of course we have that $\phi_\eps\to\phi:=\log(d(x)+\delta)-\log\delta$.  We use $\phi_\eps$ as a  test function in the equation of \eqref{fpeps}. 
Using the Neumann condition for $m_\vep$ and that $a_\vep D\phi_\vep \cdot \nu \leq 0$ on $\partial\Omega_\vep$, integrating by parts twice we obtain
\begin{equation}\label{intermediate}
\intoe m_\eps(t)\phi_\eps(t)\spazio dx+\intie m_\eps(-\mathrm{div}(a_\eps(x)D\phi_\eps)+b_\eps\, D\phi_\eps)\spazio dxds \geq \intoe m_0\phi_\eps(0)\spazio dx\,.
\end{equation}
Computing the gradient of $\phi_\vep$  we get, for $\eps$ sufficiently small,
\begin{align*}
&\intie m_\eps (\mathrm{div}(a_\eps(x)D\phi_\eps)-b_\eps\,D\phi_\eps)\spazio dxds 
\\
 & =\intie \frac{m_\eps}{d_\eps(x)+\de}\, \left\{ (\mathrm{div} (a_\eps(x)Dd_\eps)-b_\eps\, Dd_\eps) -\frac{ a_\eps(x)Dd_\eps\cdot Dd_\eps  }{d_\eps(x)+\delta}\right\}\spazio dxds
\\
  \geq 
& \intie \frac{m_\eps}{d_\eps(x)+\de}\left\{\mathrm{div}(a_\eps(x)Dd_\eps)-b_\eps\, Dd_\eps-   \frac{ a_\eps(x)Dd_\eps\cdot Dd_\eps}{d_\eps(x)+r_\eps} \right\}\spazio dxds\,.
\end{align*}
and thanks to assumption \rife{divereps}  we deduce
\begin{align*}
&\intie m_\eps (\mathrm{div}(a_\eps(x)D\phi_\eps)-b_\eps\, D\phi_\eps)\spazio dxds 
 \geq  - C \int\limits_{\Omega_\eps\cap\Gamma_{\delta_0}} m_\eps\,   \frac{d_\eps(x)+h_\eps}{d_\eps(x)+\delta}\, dxds
\\
&
+\int_0^t \int\limits_{\Omega_\eps\setminus\Gamma_{\delta_0}} \frac{m_\eps}{d_\eps(x)+\de}\,   \left\{\mathrm{div} (a_\eps(x)Dd_\eps)-b_\eps\, Dd_\eps-  \frac{ a_\eps(x)Dd_\eps\cdot Dd_\eps}{d_\eps(x)+r_\eps} \right\}\, dxds\,.
\end{align*}
The first integral in the right-hand side is uniformly bounded because $h_\vep\leq \de$ for small $\vep$. Since $b_\eps$ is locally uniformly bounded in $Q_T$, $a_\eps$ is Lipschitz and, for $\vep$ sufficiently small, $d_\eps(x)\geq\frac{\delta_0}{2}$ in $\Omega\setminus\Gamma_{\delta_0}$, the second integral is also bounded
uniformly. So we conclude that
$$
\intie m_\eps (\mathrm{div}(a_\eps(x)D\phi_\eps)-b_\eps\, D\phi_\eps)\spazio dxds 
 \geq  - C
$$ 
for some $C>0$ independent of $\vep, \de$.  Plugging this estimate into \eqref{intermediate}, we get
\be\label{pre-eps} \into m_\eps(t)\phi_\eps(t)\, dx \geq\intoe m_0\phi_\eps(0)\, dx - C\,.
\ee
Now we observe that the integral in the left side converges: indeed, for any $\eta>0$ we have (we call $\Gamma_\eta^\vep= \{x\,:\, d_\vep(x)<\eta\}$)
\begin{align*}
&  \into |m_\eps(t)\phi_\eps(t)- m(t)\phi(t)|\, dx   \leq \int_{\Gamma_\eta}  |m_\eps(t)\phi_\eps(t)- m(t)\phi(t)|\, dx
\\
& \qquad\qquad\qquad \qquad \qquad  + \int_{\Omega\setminus \Gamma_\eta} |m_\eps(t)\phi_\eps(t)- m(t)\phi(t)|dx
\\
& \quad \leq C\, \log\left(\frac{\eta+\de}\de\right) + \int_{\Omega\setminus \Gamma_\eta} |m_\eps(t)\phi_\eps(t)- m(t)\phi(t)|
dx\,.
\end{align*}
We let first $\vep \to 0$, using  the $L^1_{loc}$ convergence of $m_\eps$, and then we let $\eta\to 0$, so that last two terms will vanish. Hence we deduce
$$
\into m_\eps(t)\phi_\eps(t)\, dx \to \into m(t)\, \phi(t)\, dx\,.
$$
Therefore we obtain from \rife{pre-eps}, letting $\vep\to 0$,
$$
\into m (t)\phi(t)\, dx \geq  \into m_0\phi(0)\, dx - C \,.
$$
Since $\phi= \log(d(x)+
\de)-\log \de  \leq |\log \de| + c $ we deduce that
$$
 \into m (t)\spazio dx \geq \into m_0 \frac{\log(d(x)+
\de)-\log \de}{|\log \de|}\spazio dx - \frac C {|\log \de|}\,.
$$
We now let $\de\to 0$; using   Lebesgue's theorem and since $m_0\in \elle1$, we get
$$
\into m (t)\spazio dx \geq \into m_0\spazio dx  \,.
$$
Thus \eqref{aim} is proved, we have \rife{ae+int} and with that we conclude that $m_\vep(t) \to m(t)$ in $\elle1$, for all $t>0$.
By Lebesgue theorem, we also deduce the convergence of $m_\vep$ to $m$ in $L^1(Q_T)$.

\vskip1em

\emph{Step 3: convergence in $C^0([0,T]; L^1(\Omega))$}.

First we observe that $m\in C([0,T];L^1(\Omega))$, with a similar argument as used above. Indeed, let  $t_n\to t$; since $m\in C([0,T];L_{loc}^1(\Omega))$ we have that $m(t_n)$ always admits a subsequence converging to $m(t)$ a.e. in $\Omega$. Since $\into m(t_n)\,dx = \into m_0\, dx= \into m(t)\, dx$, we deduce again that $m(t_n) \to m(t)$ in $\elle1$.

Since $m\in C^0([0,T]; L^1(\Omega))$,  a compactness argument implies that
\be\label{equi-mt}
\lim_{|E|\to 0 }\,\,\, \sup_{t\in [0,T]} \int_E m(t) \, dx=0\,.
\ee
Using positive and negative parts, i.e. $s= s^+-s^-$, $|s|=s^++ s^- $, we split
\be\label{new1}
\begin{split}
\into |m_\vep (t)- m(t) | \, dx& = \into (m_\vep (t)- m(t) ) \, dx+ 2\into (m_\vep (t)- m(t))^- \, dx
\\
& = -\int_{\Omega\setminus \Omega_{\vep}} m_0\, dx + 2\into (m_\vep (t)- m(t))^- \, dx
\end{split}
\ee
because of mass conservation. Last integral is restricted to where $m_\vep(t)\leq m(t)$. 
So, we split once more
\begin{align*}
 \into (m_\vep (t)- m(t))^-\, dx  & \leq    \int_{\Omega\setminus\Gamma_\eta} (m_\vep (t)- m(t))^- \, dx+ 2  \int_{\Gamma_\eta} m(t) \, dx 
\\
& \leq 
 \int_{\Omega\setminus \Gamma_\eta} |m_\vep (t)- m(t)|\, dx+ 2   \int_{\Gamma_\eta} m(t) \, dx 
\end{align*}
which yields
$$
 \sup_{t\in [0,T]}\, \, \into (m_\vep (t)- m(t))^- \, dx   \leq     \sup_{t\in [0,T]}\,\, \int_{\Omega\setminus \Gamma_\eta} |m_\vep (t)- m(t)|\, dx+ 2  \sup_{t\in [0,T]}\,\,  \int_{\Gamma_\eta} m(t) \, dx \,.
$$
Now recall that $m_\vep \to m$ in $C^0([0,T]; L^1(K))$, for any compact subset $K$. So, when we let $\vep \to 0$ the first term in the right-hand side vanishes and we get
$$
\lim\limits_{\vep \to 0}\,\,  \sup_{t\in [0,T]}\, \, \into (m_\vep (t)- m(t))^- \, dx    \leq   2  \sup_{t\in [0,T]}\,\,  \int_{\Gamma_\eta} m(t) \, dx\,.
$$
Finally, we let $\eta\to 0$ and we use \rife{equi-mt} in the last term, and we conclude that
$$
\lim\limits_{\vep \to 0}\,\,  \sup_{t\in [0,T]}\, \, \into (m_\vep (t)- m(t))^- \, dx  =0\,.
$$
Then from \rife{new1} we deduce that 
$$
\lim\limits_{\vep \to 0}\,\,  \sup_{t\in [0,T]}\, \, \into |m_\vep (t)- m(t) | \, dx  =0
$$
that is $m_\vep \to m$ in $C^0([0,T]; L^1(\Omega))$.

\vskip1em			
\textit{Step 4: Conclusion.} We take $\phi\in C([0,T];L^1(\Omega))\cap L^\infty([0,T]\times\Omega)$ such that $\phi$ satisfies
$$
\begin{cases}
-\phi_t-\mathrm{div}(a(x)D\phi)+b\cdot D\phi\in L^\infty([0,T]\times\Omega)\\
\phi(T)=0
\end{cases}
$$
in the weak sense. Let us consider the solution $\phi_\vep$ of the following problem
\begin{align*}
& \qquad  \begin{cases}
-(\phi_\vep)_t-\mathrm{div}(a_\eps(x) D\phi_\vep)+b_\vep \cdot D\phi_\vep= f & \quad \hbox{in $(0,T)\times\Omega_\eps$}
\\
a_\eps(x)D\phi_\vep \cdot \nu =0 & \quad \hbox{in $(0,T)\times \partial\Omega_\eps$}
\\ \phi_\vep(T)=0 & 
\end{cases}
\\
&\hbox{where} \quad  f:= -\phi_t-\mathrm{div}(a(x)D\phi)+bD\phi\,.
\end{align*}
As always, we set $\phi_\eps:=0$ on $\Omega\setminus\Omega_\eps$.

Applying  Corollary \ref{stability}, we have that $\phi_\vep$ converges to $\phi$ in $L^2(0,T; W^{1,2}(K))$ and in $C^0([0,T]; L^p(\Omega))$ for all $p<\infty$. 
Now, taking $\phi_\vep$ as test function in \eqref{fpeps} we get
$$
 -\into m_0\phi_\vep(0)\spazio dx+ \intif m_\eps(-\phi_t-\mathrm{div}(a(x)D\phi)+b\cdot D\phi)\spazio dxdt=0\,.
$$
Since $\phi_\vep(0) \to \phi(0)$ a.e. in $\Omega$, we can pass to the limit in the first term with Lebesgue's theorem. In the second term, we use the $L^1$ convergence of $m_\vep$ towards $m$. Finally, we obtain that $m$ satisfies
$$
\intif m(-\phi_t-\mathrm{div}(a(x)D\phi)+bD\phi)\spazio dxdt=\into m_0\phi(0)\spazio dx\spazio.
$$
\end{proof}

Finally, we conclude with the existence part.

\begin{thm}\label{exfp}
Let $m_0\in   \elle1$, $m_0\ge0$.
Let $a\in W^{1,\infty}(\Omega)$ satisfy \rife{deg}.   
Assume that $  b\in L^\infty(0,T; L^\infty_{loc}( \Omega))$ and that there exist  
$\delta_0, C>0$ such that the following inequality holds:
\begin{equation}\label{diver}
\mathrm{div}(a(x)Dd(x))-b(t,x)\cdot Dd(x)\ge   \frac{a(x) Dd(x) \cdot Dd(x)}{d(x)} - C\, d(x)
\end{equation}
for almost every $t\in (0,T)$ and  $x\in\Gamma_{\delta_0}$.  
Then there exists a solution of problem \eqref{fp}.
\end{thm}

\begin{proof}  For each $\eps$ we consider $a_\eps=a$, $b_\vep=b$ and $\Omega_{\eps}=\{ d(x)>\eps \}$. It is immediate to check that  the assumptions of Proposition \ref{stab-fp} are satisfied: indeed, since $d_\eps(x)=d(x)-\eps$ near $\partial\Omega$, we have in $\Omega_{\eps}$
\begin{align*}
& \mathrm{div}(a_\eps(x)Dd_\eps(x))-b_\vep (t,x)\cdot Dd_\eps(x)=   \mathrm{div}(a(x)Dd(x))-b (t,x)\cdot Dd(x) 
\\
& \qquad \geq \frac{a(x) Dd(x) \cdot Dd(x)}{d(x)} - C\, d(x)=\frac{a_\eps(x) Dd_\eps(x) \cdot Dd_\eps(x)}{d_\eps(x)+\eps} - C\, d_\eps(x)-C\eps
\end{align*}
which gives \rife{divereps}. So by solving the approximating problems \rife{fpeps} and passing to the limit, thanks to Proposition \ref{stab-fp}, we obtain the existence of a solution.
\end{proof}

\subsection{Uniqueness of solutions}

The uniqueness of solutions for the Fokker-Planck equation comes easily from the existence of solutions of the Hamilton-Jacobi-Bellman equation.
		
\begin{thm}\label{exfp}
Let $m_0\in   \elle1$, $m_0\ge0$. Assume that $a(\cdot)\in W^{1,\infty}(\Omega)$ satisfies \eqref{deg}. Assume also that  $  b\in L^\infty(0,T; L^\infty_{loc}( \Omega))$ and that \rife{diver} holds true.
Then there exists a unique weak solution of the Fokker-Planck equation \eqref{fp}.
\end{thm}
		
\begin{proof}
Let $m_1$ and $m_2$ be two solutions of \eqref{fp}. Then $m:=m_1-m_2$ solves in the weak sense
\begin{equation*}
\begin{cases}
m_t-\mathrm{div}(a^*(x)Dm)-\mathrm{div}(m\, b)=0\\
m(0)=0\,.
\end{cases}
\end{equation*}
Now, we take $\phi$ as the solution of the Hamilton-Jacobi equation
$$
			\begin{cases}
			-\phi_t-\mathrm{div}(a(x)D\phi)+bD\phi=\mathrm{sgn}(m) \\
			\phi(T)=0\,,
			\end{cases} 
			$$
where $\mathrm{sgn}(m)= m/|m|\, \mathds 1_{\{m\neq 0\}}$.
We use $\phi$ as test function in the weak formulation of $m$, obtaining
\begin{equation*}
\intif |m|\spazio dxdt=\intif m\spazio\mathrm{sgn}(m)\spazio dxdt=0\spazio.
\end{equation*}
So $m\equiv0$ and the proof is concluded.
\end{proof}

In the end, from the equivalence between the two problems \rife{kfp} and \rife{fp}, we have proved Theorem \ref{teo2}.

%

\section{The mean field game system}

We are now ready to study the mean field game  system \eqref{mfg} under the invariance conditions. For convenience, we rewrite here the system, which reads as
 \begin{eqnarray}
  \label{mfg-hjb}
&\begin{cases} -\partial_t u - \sum\limits_{i,j} a_{ij}(x) \partial_{ij}^2 u +H(t,x,Du)=F(t,x,m)\,, &
 (t,x)\in (0,T)\times\Omega \\
u(T)= G(x,m(T))& 
\end{cases}
\\
\label{mfg-fp}
& \begin{cases}
\partial_t m - \sum\limits_{i,j} \partial_{ij}^2 (a_{ij}(x)m) -\mathrm{div}(mH_p(t,x,Du))=0\,, & (t,x)\in (0,T)\times\Omega\\
m(0)=m_0\,.& 
\end{cases}
\end{eqnarray}
Let us recall that the matrix $a(x)$ satisfies \rife{deg}, the Hamiltonian $H$ satisfies assumptions \rife{convex}-\rife{linder} and that the invariance condition \rife{invariance} holds true. The nonlinearities $F,G$ satisfy conditions \rife{F}, \rife{G}. This implies that, for any given $m\in C([0,T]\times L^1(\Omega))$, the HJB equation has a unique solution thanks to  Theorem \ref{exhjb} and Theorem \ref{unichjb}. Conversely, for every $u$ which is locally Lipschitz, the growth condition \rife{linder} guarantees that $H_p(t, x,Du)$ is a locally bounded vector field, so the FP equation has a unique solution given by Theorem \ref{exfp}. This justifies our definition below.

\begin{defn}\label{sol-mfg}
We say that a couple $(u,m)\in L^\infty([0,T]\times\Omega)\times C([0,T]\times L^1(\Omega))$ is a weak solution of the system \eqref{mfg} if $u$ is a solution of the Hamilton-Jacobi-Bellman equation \eqref{mfg-hjb} in the sense of Definition \ref{defhjb} and $m$ is a solution of the Fokker-Planck equation \eqref{mfg-fp} in the sense of Definition \ref{deffp}.
\end{defn}

\subsection{Existence of solutions}

Here we prove the existence of a  solution to the mean field game system.

\begin{thm}\label{existence}
Assume that hypotheses \rife{deg}, \rife{convex}-\rife{linder}, \rife{invariance}, \rife{F} and \rife{G} hold true. 
Then there exists at least one solution $(u,m)$ of  \eqref{mfg-hjb}-\rife{mfg-fp}, in the sense of Definition  \ref{sol-mfg}.
\end{thm}

\begin{proof}
For each $\eps>0$, we define $\Omega_{\eps}=\{ d(x)>\eps \}$ and $(u_\eps,m_\eps)$ as the solution, in $[0,T]\times \Omega_\eps$, of the mean-field game system
\begin{equation}\label{mfgeps}
\begin{cases}
-\partial_t u_\eps - \mathrm{tr}(a(x)D^2u_\eps)+H(t,x,Du_\eps)=F(x,m_\eps)\,,\qquad(t,x)\in (0,T)\times\Omega_\vep \\
\partial_t m_\eps - \mathrm{div}(a^*(x)Dm_\eps)-\mathrm{div}(m_\eps (H_p(t,x,Du_\eps)+\tilde{b}(x))=0\, \quad (t,x)\in (0,T)\times\Omega_\vep \\
m_\eps(0)=m_0\,, \hspace{2cm} u_\eps(T)=G(x,m_\eps(T))\\
a(x)Du_\eps\cdot\nu_{|\partial\Omega_\eps}=0\hspace{1.5cm}\left[a(x)Dm_\eps + m_\eps(\tilde{b}(x)+H_p(t,x,Du_\eps))\right]\cdot\nu_{|\partial \Omega_\eps}=0\,.
\end{cases} 
\end{equation}
As before, we extend the solutions to the whole of $\Omega$ by setting $u_\eps=m_\eps=0$ in $\Omega\setminus\Omega_\eps$.

By conservation of mass, we have that $\int_{\Omega_\vep} m_\vep(t)=\into\, m_0\, dx$ for all $t\in (0,T)$, and additionally $m_\vep\geq 0$. Then,  assumptions \rife{F}, \rife{G} imply that $F(x,m_\eps)$ and $G(x,m_\eps(T))$ are uniformly bounded.
By maximum principle, we deduce that $\|u_\vep\|_{\infty}$ is uniformly bounded. Applying Lemma \ref{compact-hjb}, we deduce that  there exists a function $u\in L^\infty(Q_T) \cap L^2(0,T; W^{1,2}_{loc}(\Omega))$ and   a subsequence $u_\vep$ converging to $u$ weakly in $L^2(0,T; W^{1,2}(K))$, for all compact sets $K\subset \Omega$. 
Moreover, the assumption upon $G$, the global bound on $u_\vep$ and the natural growth conditions ensure that the sequence $Du_\vep$ is also bounded in any  set $(0,T)\times K$,  for $K$ compact. This allows us to use Lemma \ref{compact-fp} for $m_\vep$.  
In fact, since $H_p(t,x,Du_\vep)$ is locally bounded and converges a.e. to $H_p(t,x,Du)$, and since the invariance condition \rife{invariance} holds, we are in the position to apply the stability result of Proposition \ref{stab-fp} as well. Therefore, we conclude that 
$$
m_\vep \to m \qquad \hbox{in $C^0([0,T]; L^1(\Omega))$}
$$
and $m$ is a solution of \rife{mfg-fp}. Finally, the continuity assumptions upon $F$ and $G$ now imply that  $F(t,x,m_\vep)$  converges almost everywhere  to $F(t,x,m)$ in $Q_T$ and  $G(x,m_\vep(T))$ converges to $G(x,m(T))$ a.e. and therefore in $L^p$ for all $p<\infty$. We have now access to the stability result of Proposition \ref{stab-hjb} and we deduce that the limit function $u$ is a solution of \rife{mfg-hjb}. 
\end{proof}

\subsection{Uniqueness of solutions}
\begin{thm} Suppose the hypotheses of Theorem \ref{existence} are satisfied and, in addition, $F$ and $G$ are nondecreasing with respect to $m$, in the sense of operators. If at least one of the two following conditions holds:
\be\label{strict-mon}
\begin{split}
(i) &\quad  \begin{cases} \into (F(t,x,m_0)-F(t,x,m_1))d(m_0-m_1)=0 \,\,\Rightarrow \,\, F(t,x,m_0)=F(t,x,m_1)\,
& \\
\into (G(x,m_0)-G(x,m_1))d(m_0-m_1)=0 \,\,\Rightarrow \,\, G(x,m_0)=G(x,m_1)
& \end{cases}
\\
\m
(ii) & \quad H(t,x,p_1)-H(t,x,p_2)-H_p(t,x,p_2)(p_1-p_2)=0   \,\,\Rightarrow \,\, H_p(t,x,p_1)=H_p(t,x,p_2)\,.
\end{split}
\ee
then the solution of  \eqref{mfg} is unique.
\end{thm}

\begin{proof}
Let $(u,m)$ and $(v,\mu)$ be two solutions of the mean field game system. We want to prove that $v=u$, $\mu=m$.\\
To do this, we reason as always through approximation. Having defined  $\Omega_\vep$   as in Theorem \ref{existence}, we consider $(u_\eps,m_\eps)$   solution of the problem
\begin{equation}\label{mfgapprox1}
\begin{cases}
-\partial_t u_\eps - \mathrm{tr}(a(x) D^2u_\eps)+{H}(t,x,Du_\eps)=F(t,x,m)\qquad (t,x)\in (0,T)\times \Omega_\vep\\
\partial_t m_\eps - \mathrm{div}(a^*(x) Dm_\eps)-\mathrm{div}(m_\eps (H_p(t,x,Du_\eps)+\tilde{b}(x))=0\qquad (t,x)\in (0,T)\times \Omega_\vep\\
m_\eps(0)=m_0 \hspace{2cm} u_\eps(T)=G(x,m(T))\\
Du_\eps\cdot\nu_{|\partial\Omega_\vep}=0\hspace{1.5cm}\left[\eps Dm_\eps + m_\eps(\tilde{b}(x)+H_p(t,x,Du_\eps))\right]\cdot\nu_{|\partial\Omega_\vep}=0
\end{cases}\,.
\end{equation}
Similarly, $(v_\eps,\mu_\eps)$ will be the solution of the problem
\begin{equation}\label{mfgapprox2}
\begin{cases}
-\partial_t v_\eps - \mathrm{tr}(a(x) D^2v_\eps)+H(t,x,Dv_\eps)=F(t,x,\mu)\qquad (t,x)\in (0,T)\times \Omega_\vep\\
\partial_t \mu_\eps - \mathrm{div}(a^*(x) D\mu_\eps)-\mathrm{div}(\mu_\eps (H_p(t,x,Dv_\eps)+\tilde{b}(x))=0\qquad (t,x)\in (0,T)\times \Omega_\vep\\
\mu_\eps(0)=m_0 \hspace{2cm} v_\eps(T)=G(x,\mu(T))\\
Dv_\eps\cdot\nu_{|\partial\Omega_\vep}=0\hspace{1.5cm}\left[\eps D\mu_\eps + \mu_\eps(\tilde{b}(x)+H_p(t,x,Dv_\eps))\right]\cdot\nu_{|\partial\Omega_\vep}=0\,.
\end{cases} 
\end{equation}
Notice that the equations of $u_\eps$ and $v_\eps$ are \emph{decoupled} from the system, since they do not depend, respectively,  upon $m_\eps$ and $\mu_\eps$. Using Corollary \ref{stability}, we know that $u_\eps\to u$ and $v_\eps\to v$ in $C([0,T];L^p(K))$ and in $L^2([0,T];H^1(K))$, for each compact $K\subset\subset\Omega$.

Using this information, and the local gradient bounds, we know that $H_p(x,Du_\eps)$ and $H_p(x,Dv_\eps)$ are locally bounded sequences which converge, respectively, to $H_p(x,Du)$ and $H_p(x,Dv)$. From Proposition \ref{stab-fp} we deduce that  $m_\eps\to m$ and $\mu_\eps\to\mu$ in $C([0,T];L^1(\Omega))$.

Now we use the classical monotonicity argument in mean field game systems. We estimate in two different ways the quantity
			\begin{equation*}
			\int_0^T \int_{\Omega_\vep}\left((u_\eps-v_\eps)(m_\eps-\mu_\eps)\right)_t\spazio dxdt\spazio.
			\end{equation*}
First, computing directly the time integral we find
\begin{equation*}
\int_0^T \int_{\Omega_\vep}\left((u_\eps-v_\eps)(m_\eps-\mu_\eps)\right)_t\spazio dxdt=  \int_{\Omega_\vep} (G(x,m(T))-G(x,\mu(T)))(m_\eps(T)-\mu_\eps(T))\spazio dx\spazio.
\end{equation*}
Besides, if we  use the weak formulations of $u_\eps$, $v_\eps$, $m_\eps$, $\mu_\eps$, we obtain 
\begin{align*}
&\int_0^T \int_{\Omega_\vep}\left((u_\eps-v_\eps)(m_\eps-\mu_\eps)\right)_t\spazio dxdt=-\int_0^T \int_{\Omega_\vep} (F(x,m)-F(x,\mu))(m_\eps-\mu_\eps)\spazio dxdt-\\-&\int_0^T \int_{\Omega_\vep} m_\eps(H(t,x,Dv_\eps)-H(t,x,Du_\eps)+H_p(x,Du_\eps)(Dv_\eps-Du_\eps))\spazio dxdt-\\-&\int_0^T \int_{\Omega_\vep} \mu_\eps(H(t,x,Du_\eps)-H(t,x,Dv_\eps)+H_p(x,Dv_\eps)(Du_\eps-Dv_\eps))\spazio dxdt\spazio 
\end{align*}
which yields
\begin{align*}
& \int_{\Omega_\vep} \left[G(x,m(T))-G(x,\mu(T))\right] (m_\eps(T)-\mu_\eps(T))\spazio dx + \int_0^T \int_{\Omega_\vep} (F(x,m)-F(x,\mu))(m_\eps-\mu_\eps)\spazio dxdt\\+&\int_0^T \int_{\Omega_\vep} m_\eps(H(t,x,Dv_\eps)-H(t,x,Du_\eps)+H_p(x,Du_\eps)(Dv_\eps-Du_\eps))\spazio dxdt\\+&\int_0^T \int_{\Omega_\vep} \mu_\eps(H(t,x,Du_\eps)-H(t,x,Dv_\eps)+H_p(x,Dv_\eps)(Du_\eps-Dv_\eps))\spazio dxdt\leq 0\,.
\end{align*}
Since $H$ is convex, we can apply Fatou's lemma in the last two integrals. Moreover, using that  (extending the functions to zero outside $\Omega_\vep$) $m_\eps(T)-\mu_\eps(T)\to m(T)-\mu(T)$ in $L^1(\Omega)$, we can pass to the limit in the remaining two integrals. We obtain
\begin{align*}
& \into \left[G(x,m(T))-G(x,\mu(T))\right] (m (T)-\mu (T))\spazio dx + \int_0^T \into (F(x,m)-F(x,\mu))(m -\mu )\spazio dxdt\\+&\int_0^T \into m (H(t,x,Dv )-H(t,x,Du )+H_p(x,Du )(Dv -Du ))\spazio dxdt\\+&\int_0^T \into \mu (H(t,x,Du )-H(t,x,Dv )+H_p(x,Dv)(Du -Dv ))\spazio dxdt\leq 0\,,
\end{align*}
and therefore all integrals must vanish. Now we conclude with assumption \rife{strict-mon}. Indeed, if (i) holds we deduce that 
$$
F(t,x,m)=F(t,x,\mu)\spazio,\hspace{2cm}G(x,m(T))=G(x,\mu(T))\spazio.
$$
This means that $u$ and $v$ solve the same Hamilton-Jacobi-Bellman equation. From Theorem \ref{unichjb}, we know that they coincide. So $v=u$, hence  $H_p(x,Du)=H_p(x,Dv)$. Coming back to the Fokker-Planck equation, we deduce that $\mu=m$ from Theorem \ref{exfp}. Otherwise, if (ii) holds, we proceed in the opposite way: first we deduce that $H_p(x,Du)=H_p(x,Dv)$, and then $\mu=m$, which in turn implies that $u=v$ by uniqueness of the HJB equation.
\end{proof}

\begin{rem}\label{rem-couplings} We stress that the assumption on the final pay-off $G$ can be relaxed in case that $H(t,x,p)$ is globally Lipschitz continuous with respect to $p$. Indeed, in this case the drift term  $H_p(t,x,Du)$ is always bounded and it is not needed that the range of $G$ be bounded in $W^{1,\infty}(\Omega)$. It would be enough, in this case, to require that the range of $G$ is bounded in $\elle\infty$, similar as it is done for the internal coupling $F$. In particular, this condition would include local couplings, i.e.  $G=G(x,r)$ is a bounded real function. 
%
\end{rem}

\section{Further regularity of solutions}

In this Section we get an improvement of regularity for $u$ or $m$ with a suitable strengthening of hypotheses.

\subsection{Lipschitz regularity of the value function}

We follow the classical Bernstein method in order to get gradient bounds for the solution of $u$. The approach is borrowed from \cite{LP} and yields the global Lipschitz character of the value function.

\begin{thm}\label{lipthm}

Assume that $a(x)$ satisfies \rife{deg} and there exists a matrix $\sigma\in W^{1,\infty}(\Omega)$ such that $a(x)=\sigma(x)\sigma(x)^*$. Let $H\in C^1(Q_T\times\mathbb{R}^N)$ satisfy
conditions \rife{convex}-\rife{linder} and, in addition, the following assumption:
\be\label{Hx}
 H_x(t,x,p) \cdot p\geq - C\, (1+ |p|^2)\qquad \forall (t,x)\in Q_T\,,p\in \R^N 
\ee
for some constant $C>0$. Moreover, assume that the invariance condition \rife{invariance} holds true.
Let $F,G$ satisfy \rife{F}, \rife{G} and assume that $m\mapsto F(\cdot ,m)$  has bounded range  in $L^\infty((0,T);W^{1,\infty}(\Omega))$.
Then $ u \in L^\infty((0,T);W^{1,\infty}(\Omega))$.
\end{thm}

\begin{proof}
Let $\Omega^\vep= \{x\,:d(x) >\vep\}$. We consider $u_\eps$ solution of the problem
\begin{equation}\label{hjbeps2}
\begin{cases}
-(u_\eps)_t-\mathrm{tr}( a(x) D^2u_\eps)+H(t,x,Du_\eps) =F(x,m)\,,\qquad (t,x)\in (0,T)\times \Omega^\vep\\
u_\eps(T)=G(x,m(T))\\
Du_\eps\cdot\nu_{|\partial\Omega^\vep}=0\,.
\end{cases}\spazio.
\end{equation}
We know from Proposition \ref{altra-app-hjb} that $u_\eps\to u$ when $\eps\to 0$, with $Du_\eps\to Du$ a.e. in $ Q_T$.

Let us set $w_\eps:=|Du_\eps|^2e^{\theta(d(x))}$, where $\theta\in C^2(0,1)$ is a bounded function to be  defined later.
Computing the derivatives of $w_\eps$, we find:
\begin{align*}
Dw_\eps&= e^{\theta(d)}\left(2Du_\eps D^2u_\eps+|Du_\eps|^2\theta'(d)Dd\right)\spazio;\\	
D^2w_\eps&=\etd\left(2D^2u_\eps D^2u_\eps+2Du_\eps D^3u_\eps+|Du_\eps|^2\theta'(d)D^2d+4\theta'(d)D^2u_\eps Du_\eps\otimes Dd \right)+\\&+\etd\, |Du_\eps|^2 [\theta''(d)+ (\theta'(d))^2] Dd \otimes Dd \spazio
\end{align*}
where $Du_\eps D^3 u_\eps=\sum\limits_{k}(u_\eps)_{x_k}(u_\eps)_{x_i x_j x_k}$. 

Thus when we form the equation for $w_\vep$ we obtain (see also \cite[Lemma 8]{LP})
%
%
\begin{equation*}
-(w_\eps)_t-\mathrm{tr}( a(x) D^2w_\eps)+ b_\eps(t,x)Dw_\eps+c_\eps(t,x)w_\eps=r_\eps(t,x)\spazio,
\end{equation*}
where
\begin{align*}
b_\eps(t,x)&= H_p(x,Du_\eps)+2\theta'(d)\left(a(x)Dd\right);\\
c_\eps(t,x)&= \theta'(d)\left(-H_p(x,Du_\eps)\cdot Dd +\mathrm{tr}(a(x) D^2d)\right) +\left(\theta''(d)-\left(\theta'(d)\right)^2 \right)\, a(x) Dd\cdot Dd ;\\
r_\eps(t,x)&= 2\etd \left(\mathrm{tr}(\tilde{a}(u_\eps)D^2u_\eps)-\mathrm{tr}(a(x)D^2u_\eps D^2u_\eps)-  H_x(x,Du_\eps)\cdot Du_\eps+ DF\,\cdot Du_\vep \right)
\end{align*}
and we denoted  $\tilde{a}(u_\eps)_{i,j}=\sum\limits_{k}(a_{i,j}(x))_{x_k}(u_\eps)_{x_k}$.
\vskip1em
First we estimate the quantity $\mathrm{tr}(\tilde{a}(u_\eps)D^2u_\eps)-\mathrm{tr}(a(x) D^2u_\eps D^2u_\eps)$. Here, since  $a(\cdot)=\sigma(\cdot)\sigma(\cdot)^*$, using Young's inequality we get
\begin{align*}
&\mathrm{tr}(\tilde{a}(u_\eps)D^2u_\eps)-\mathrm{tr}(a(x) D^2u_\eps D^2u_\eps)=\\
=&\sum\limits_{i,j,k}(a_{ij})_{x_k}(u_\eps)_{x_k}(u_\eps)_{x_i x_j}-\sum\limits_{i,j,k}a_{ij}(u_\eps)_{x_j x_k}(u_\eps)_{x_i x_k}=\\=&
2\sum\limits_{i,j,k,l}(\sigma_\eps)_{jl}((\sigma_\eps)_{il})_{x_k}(u_\eps)_{x_k}(u_\eps)_{x_i x_j}-\sum\limits_{k}\left|\sigma^{*}_\vep D(u_\eps)_{x_k}\right|^2\le C|Du_\eps|^2\spazio.
\end{align*}
Using also \rife{Hx} and the condition on $F$, we estimate
$$
r_\eps \leq C\, \etd \left( 1+|Du_\eps|^2\right)\,.
$$
Therefore, we have that $w_\eps$ satisfies
\begin{equation*}
-(w_\eps)_t-\mathrm{tr}(a(x) D^2w_\eps)+b_\eps(t,x)Dw_\eps+(c_\eps(t,x)-C)w_\eps\le C,
\end{equation*}
for a suitable $C>0$. 

Now we estimate $c_\vep$ thanks to the invariance condition \rife{invariance}. Indeed, if $d(x)<\de_0$, we get
$$
c_\eps  \geq  \left(\frac{\theta'(d)}d+\theta''(d)-\left(\theta'(d)\right)^2 \right)\, a(x) Dd\cdot Dd - C\, d\, \theta'(d)\,.
$$
Choosing $\theta(d)=d^\gamma$, with $\gamma\in(0,1)$, we get
$$
c_\eps  \geq  \gamma\, d^{\gamma-1} \left(\frac1d+(\ga-1) d^{-1}-\gamma\, d^{\gamma-1} \right)\, a(x) Dd\cdot Dd - C\, \ga\,d^{\ga}\,
$$
hence $c_\vep$ is uniformly bounded below. If $d(x)\geq \de_0$, we recall that $u$ is Lipschitz by elliptic regularity, so $c_\vep$  is also bounded from below by a  constant possibly dependent on $\de_0$, but independent from $\eps$. We conclude that
$w_\vep$ satisfies
$$
-(w_\eps)_t-\mathrm{tr}(a(x) D^2w_\eps)+b_\eps(t,x)Dw_\eps -C \, w_\eps\le C\,.
$$
Since the maximum of $w_\vep$ cannot be taken on the boundary due to the Neumann condition (see e.g. \cite[Lemma 4]{LP}), and since at $t=T$ we use the Lipschitz bound on $G(\cdot, m(T))$, we conclude applying the maximum principle that 
the maximum of $Dw_\eps$ is uniformly bounded in $\eps$. As $\vep\to 0$, this implies that $Du\in L^\infty(Q_T)$.
\end{proof}

\begin{rem} We stress that the above proof may admit some variants which possibly apply to 
other interesting cases. For instance, assume that  the invariance condition is strengthened as follows:
\begin{equation}\label{dtheta}
\mathrm{tr}(a(x)D^2d(x))-H_p(x,p)Dd(x)\ge \frac{a(x) Dd(x) \cdot Dd(x)}{d(x)^{1+\rho}}- C d(x)\,,
\end{equation}
for some $\rho>0$, and in addition that $a(\cdot)=\sigma(\cdot)\sigma(\cdot)^*$ with 
\be\label{sig}
|D\sigma(x)|^2 \leq c_0 + c_1 |\sigma(x)Dd(x)|^2\, d(x)^{\gamma-2-\rho}
\ee
for some $\gamma>0$.  Then the conclusion of Theorem \ref{lipthm} remains true, and the proof can be easily modified accordingly. 

In particular, whenever \rife{dtheta} is satisfied,  this generalization includes the case  that $\sigma(x)$ is $\frac12$-H\"older continuous with $|\sigma(x)Dd(x)|\geq c\, \, d(x)^{\frac12}$ and $|D\sigma(x)|\leq C\, d(x)^{-\frac12}$, in which case \rife{sig} holds for any $\gamma<\rho$. Otherwise, \rife{dtheta}--\rife{sig} are satisfied if $\sigma(x)$ is $\beta$-H\"older continuous, $\beta>\frac12$, $|\sigma(x)Dd(x)|\geq c\, \, d(x)^{\beta}$ and $|D\sigma(x)|\leq C\, d(x)^{\beta-1}$ and the Hamiltonian satisfies, in a neighborhood of the boundary, that
$$
\mathrm{tr}(a(x)D^2d(x))-H_p(x,p)Dd(x)\ge c\, d(x)^{\eta}
$$
for some $\eta<2\beta-1$. An assumption of this kind appears for instance in \cite{BCR}, \cite{CCR}.
\end{rem}

\begin{rem}
An assumption as \rife{Hx} may not allow the application  to Hamiltonians with super linear growth and inhomogeneous coefficients. However, we point out that  more general conditions on the growth of the Hamiltonian could still lead to the Lipschitz bound, exactly as it is done in \cite[Theorem 4]{LP}. The strategy in that case is to use a change of unknown (typically of exponential type) in addition to the usual Bernstein method. However this leads to an increase of technicalities which we decided to omit here, for the sake of brevity. 
\end{rem}

\subsection{Semiconcavity of the value function}

If we require  stronger assumptions, we can also prove  a semi-concavity bound on $u$. This will be helpful to improve the regularity of $m$ under suitable assumptions.\\

We recall that a function $f$ is said to be \emph{semiconcave} in $\Omega$ if $\exists \,C>0$ such that
\begin{align}
\label{semiconcave}
f(x+h)+f(x-h)-2f(x)\le C|h|^2\spazio,
\end{align}
for each $x\in\Omega$, $h\in\R^N$ such that $x+h, x-h\in\Omega$.
			
In order to prove that $u$ is semi concave, we will follow the tripling variable method used in \cite{ishii}. To this purpose,  we define the following function, that will play a crucial role:
$$
\psi(x,y,z)=|x-z|^4+|y-z|^4+2|x+y-2z|^2\spazio.
$$
Then the semi-concavity of $f$ is true if the following relation holds:
$$
f(x)+f(y)-2f(z)\le C\sqrt{\psi(x,y,z)}\spazio,\hspace{1.5cm}\forall x,y,z\in\Omega\spazio.
$$
Indeed, it suffices to take $x=x'+h$, $y=x'-h$, $z=x'$ to obtain \eqref{semiconcave}. We also recall that an equivalent formulation of this latter condition is the following: there exists a constant $C>0$ such that 
\begin{align*}
f(x)+f(y)-2f(z)\le C\left(\delta+\frac{\psi(x,y,z)}{\delta}\right)\quad  \forall \delta>0\, \, \forall x,y,z\in\Omega\spazio. 
\end{align*}
Moreover, it is well-known that a function $f$ is in $W^{2,\infty}(\Omega)$ if and only if $\exists C>0$ such that
$$
|f(x)+f(y)-2f(z)|\le C\sqrt{\psi}\spazio,
$$
or, equivalently, $\forall \delta>0$
\begin{align}\label{w2inf}
|f(x)+f(y)-2f(z)|\le C\left(\delta+\frac{\psi(x,y,z)}{\delta}\right)\spazio.
\end{align}
\begin{thm}\label{usemiconc}
Suppose the hypotheses of Theorem \ref{lipthm} are satisfied. Moreover, suppose that $a(x)= \sigma(x)\sigma(x)^*$, with $\sigma\in{W}^{2,\infty}(\Omega)$ and that  $F(\cdot,m)\in W^{2,\infty}(\Omega)$, $G(\cdot,m)\in W^{2,\infty}({\Omega})$ uniformly with respect to $m$. 
Finally, we require the following hypothesis on $H$: there exist constants $C_0, C_1$ such that
\be\label{hypH}
\begin{split}
H(t,x,p) & +H(t,y,q)-2H\left(t,z,\frac{p+q}2\right)   \geq 
\\
& \quad  - C_0 (|x-z|^2+ |y-z|^2+|x+y-2z| ) (1+ |p+q| ) -C_1  \,|x-y|\, |p-q|  
\end{split}
\ee				
for any $(x,y,z)\in\Omega$, $(p,q)\in\mathbb{R}^N$, $t\in (0,T)$.
 \vskip0.4em
Then $u(t,\cdot)$ is a semiconcave function for all $t\in[0,T]$, with a semiconcavity constant bounded uniformly for $t\in (0,T)$. Namely,  we have
$$
D^2u(t) \leq M\qquad \forall t\in (0,T)\,,
$$
where $M$ depends on $T$, $\|D^2\sigma\|_\infty, \|D^2 F\|_\infty, \|D^2 G\|_\infty$ and on $H$ (through the constants appearing in the growth  conditions).
\end{thm}

\begin{proof}
We closely follow the proof given  in  \cite[Theorem VII.3]{ishii} with two main novelties:  the boundary contribution, which will be handled through the invariance condition, and the structure condition \rife{hypH}  rather than the case of pure  Bellman operators.  

In the end, we wish to prove that there exist, $ M>0$ such that
\begin{align}\label{semicu}
u(t,x)+u(t,y)-2u(t,z)\le M\left(\delta+\frac{1}{\delta}\psi(x,y,z)\right)\spazio,
\end{align}
for every $ t\in[0,T]$, every  $x, y, z\in\Omega$ and for any $\de>0$ sufficiently small.
\\
For a given $k>0$ we consider the function $v(t,x)=e^{-k(T-t)}u(t,x)$. It satisfies the parabolic equation
				\begin{equation}\label{eqv}
				\begin{cases}
				-v_t-\mathrm{tr}(a(x)D^2v)+\tilde{H}(t,x,Dv)+kv=\tilde{F}(t,x,m)\\
				v(T)=G(x,m(T))\,,
				\end{cases} 
				\end{equation}
				with
				$$
				\tilde{H}(t,x,p)=e^{-k(T-t)}H(t,x,e^{k(T-t)}p)\spazio,\hspace{1cm}\tilde{F}(t,x,m)=e^{-k(T-t)}F(x,m)\spazio.
				$$
We note that $\tilde{H}$ satisfies \eqref{hypH} and \eqref{invariance} uniformly in $t$, $k$, and that $\tilde{F}(t,\cdot,m)\in W^{2,\infty}(\Omega)$ uniformly in $t$.
For $(\gamma,\delta,M)\in(0,+\infty)^3$ we take the following function:
$$
\varphi(x,y,z)=M\left(\delta+\frac{\psi(x,y,z)}{\delta}\right)-\gamma\log(d(x)d(y)d(z)) 
$$			
where, without loss of generality, we assume that $d(x)\leq 1$ in $\Omega$. Hence $\log(d(x)d(y)d(z))\leq 0$. 
As usual, we assume  that 
\begin{align*}
\sup\limits_{[0,T]\times\Omega^3}\mathlarger{(}v(t,x)+v(t,y)-2v(t,z)-\varphi(x,y,z)\mathlarger{)}>0\,,
\end{align*}	
and we will reach a contradiction if $k,M$ are sufficiently large and $\gamma$ sufficiently small, independently of the choice of $\de$.

Since $\phi(x,y,z)=+\infty$ if one of $x$, $y$, $z$ lies in $\partial\Omega$, the $\sup$ is attained at a point $(\bar{t},\bar{x},\bar{y},\bar{z})= (\bar{t}_{\delta,M,\gamma, k},\bar{x}_{\delta,M,\gamma, k},\bar{y}_{\delta,M,\gamma, k},\bar{z}_{\delta,M,\gamma,k})\in[0,T]\times\Omega^3$. We drop the indexes for simplicity of notation. 
We observe further that, if $\bar{t}=T$, then we have
\begin{align*}
v(t,x)+v(t,y)&-2v(t,z)-\varphi(x,y,z)= \\ &=  G(x,m(T))+G(y,m(T))-2G(z,m(T))-\varphi(x,y,z)\spazio,
\end{align*}
which implies, thanks to the regularity of $G$,
\begin{align*}
v(t,x)+v(t,y)&-2v(t,z)\le (C-M)\left(\delta+\frac{\psi({x},{y},{z})}{\delta}\right)\leq 0
\end{align*}
provided  $M\ge C$. So the sup must be attained at $\bar t<T$. 

Now we proceed with typical viscosity solutions' arguments. Indeed, standing on the uniqueness result, it is easy to see that $v$ is also a viscosity solution. It may actually be  the case that $v$ is smooth inside the domain, but we prefer to keep the argument in viscosity sense for a possibly wider generality. 
By Jensen's lemma, there exists matrices $X,Y,Z$ and scalars $a,b,c$ such that
\be\label{matrix}
\begin{pmatrix}
 X & 0   & 0  \\
\noalign{\medskip}
0 & Y  & 0
 \\
\noalign{\medskip}
0& 0& Z
\end{pmatrix}
\leq D^2 \vfi(\bar x, \bar y, \bar z)
\ee
and
\begin{align*}
-a-\mathrm{tr}(a(\bar x) X)+\tilde{H}(\bar t,\bar x,D_x \vfi)+kv(\bar t, \bar x) \leq \tilde{F}(\bar t,\bar x,m(\bar t))
\\
\m
-b-\mathrm{tr}(a(\bar y) Y)+\tilde{H}(\bar t,\bar y,D_y \vfi)+kv(\bar t, \bar y) \leq \tilde{F}(\bar t,\bar y,m(\bar t))
\\
\m
-  c-\mathrm{tr}\left(a(\bar z) \left(-\frac12 Z\right) \right)+\tilde{H}(\bar t,\bar z, -\frac12 D_z \vfi)+kv(\bar t, \bar z) \geq \tilde{F}(\bar t,\bar z,m(\bar t))
\end{align*}
where $\vfi$ is computed at $(\bar t, \bar x, \bar y, \bar z)$ and where $a,b,c$ (the {\it time derivatives} in viscosity sense)   are real numbers such that $a+b\leq 2c$.
We multiply by $2$ the  latter inequality, we sum  and we get
\begin{align*}
& -\mathrm{tr}(a(\bar x) X+a(\bar y) Y + a(\bar z)  Z) + 
k \{ v(\bar t, \bar x)+ v(\bar t, \bar y)- 2v(\bar t, \bar z) \} 
\\
& \quad 
+ \hat{H}(\bar t,\bar x,\bar y,\bar z)
\leq \hat{F}(\bar t,\bar x,\bar y,\bar z)
\end{align*}
where 
\begin{align*}
\hat{H}(t,x,y,z)=\tilde{H}(t,x,D_x\vfi(t,x))+\tilde{H}(t,y,D_y\vfi(t,y))-2\tilde{H}(t,z,-\frac12 D_z\vfi(t,z))\spazio,\\
\hat{F}(t,x,y,z)=\tilde{F}(t,x,m(t))+\tilde{F}(t,y,m(t))-2\tilde{F}(t,z,m(t))\spazio.
\end{align*}
We multiply inequality \rife{matrix} by the matrix
$\Sigma=\Sigma(x,y,z)$, which is defined (in blocks) as 
				\[
				\Sigma(x,y,z)=
				\begin{pmatrix}
				\sigma(x)\sigma^*(x) & \sigma(x)\sigma^*(y) & \sigma(x)\sigma^*(z)\\
				\sigma(y)\sigma^*(x) & \sigma(y)\sigma^*(y) & \sigma(y)\sigma^*(z)\\
				\sigma(z)\sigma^*(x) & \sigma(z)\sigma^*(y) & \sigma(z)\sigma^*(z)
				\end{pmatrix}\spazio,
				\]
so we estimate
$$
\mathrm{tr}(a(\bar x) X+a(\bar y) Y + a(\bar z)  Z) \leq \mathrm{tr}\left(\Sigma(\bar{x},\bar{y},\bar{z}) D^2\varphi(\bar{x},\bar{y},\bar{z})\right)\,.
$$
We also estimate, since $	\bar{t},\bar{x},\bar{y},\bar{z}$ is a maximum point and the maximum is positive
$$
v(\bar t, \bar x)+ v(\bar t, \bar y)- 2v(\bar t, \bar z) \geq \varphi(\bar{x},\bar{y},\bar{z})\,.
$$			
Finally, we deduce that $\vfi$ satisfies
\begin{align*} 
k\varphi(\bar{x},\bar{y},\bar{z})\le \hat{F}(t,x,y,z)+\mathrm{tr}\left(\Sigma(\bar{x},\bar{y},\bar{z}) D^2\varphi(\bar{x},\bar{y},\bar{z})\right)-\hat{H}(\bar{t},\bar{x},\bar{y},\bar{z})\spazio.
\end{align*}
Using the $W^{2,\infty}$ regularity of $F$, which therefore satisfies \rife{w2inf}, we get
\begin{align}\label{apfuma}
k\varphi(\bar{x},\bar{y},\bar{z})\le C\conc+\mathrm{tr}\left(\Sigma(\bar{x},\bar{y},\bar{z}) D^2\varphi(\bar{x},\bar{y},\bar{z})\right)-\hat{H}(\bar{t},\bar{x},\bar{y},\bar{z})\spazio.
\end{align}
We analyse the two latter terms. From now on, we will omit the dependences from $(\bar{t},\bar{x},\bar{y},\bar{z})$ when there will be no possible mistake. We have
				\begin{align*}
				\mathrm{tr}(\Sigma D^2\varphi)=\frac M\delta\mathrm{tr}(\Sigma D^2\psi) -\gamma \mathrm{tr}\left(\Sigma D^2\left(\log(d(x)d(y)d(z))\right)_{|(x,y,z)=(\bar{x},\bar{y},\bar{z})}\right)\spazio.
				\end{align*}
				So, we start computing the Hessian matrix of the function $\psi$. We get
\begin{align}\label{lederivatedipsi}
D_{x}\psi=4|\bar{x}-\bar{z}|^2(\bx-\bz)+4(\bx+\by-2\bz)\spazio,\\
D_y\psi=4|\by-\bz|^2(\by-\bz)+4(\bx+\by-2\bz)\spazio,\\
D_z\psi=-4|\bx-\bz|^2(\bx-\bz)-4|\by-\bz|^2(\by-\bz)-8(\bx+\by-2\bz)\spazio,
\end{align}
and so
\begin{align*}
&D^2_{xx}\psi=8(\bx-\bz)\otimes(\bx-\bz)+4|\bx-\bz|^2I+4I\spazio,\hspace{1cm}D^2_{xy}\psi=4I\spazio,\\
& D^2_{xz}\psi=-8(\bx-\bz)\otimes(\bx-\bz)-4|\bx-\bz|^2I-8I\spazio,\\
&D^2_{yy}\psi=8(\by-\bz)\otimes(\by-\bz)+4|\by-\bz|^2I+4I\spazio,\\
&D^2_{yz}\psi=-8(\by-\bz)\otimes(\by-\bz)-4|\by-\bz|^2I-8I\spazio,\\
&D^2_{zz}\psi=8(\bx-\bz)\otimes(\bx-\bz)+4|\bx-\bz|^2I+8(\by-\bz)\otimes(\by-\bz)+4|\by-\bz|^2I+16I\spazio,
\end{align*}
Therefore, computing the first trace we found
\begin{align*}
				\frac M\delta\mathrm{tr}(\Sigma D^2\psi)=\spazio&4\frac M\delta\mathrm{tr}((\sigma(\bx)+\sigma(\by)-2\sigma(\bz))(\sigma^*(\bx)+\sigma^*(\by)-2^*\sigma(\bz))+\\+\spazio&4\frac M\delta|\bx-\bz|^2\,\mathrm{tr}\left((\sigma(\bx)-\sigma(\bz))(\sigma^*(\bx)-\sigma^*(\bz))\right)+\\+\spazio&4\frac M\delta|\by-\bz|^2\, \mathrm{tr}\left((\sigma(\by)-\sigma(\bz))(\sigma^*(\by)-\sigma^*(\bz))\right)+\\+\spazio&8\frac M\delta|(\sigma^*(\bx)-\sigma^*(\bz))(\bx-\bz)|^2+8\frac M\delta|(\sigma^*(\by)-\sigma^*(\bz))(\by-\bz)|^2\spazio.
				\end{align*}
Using the $W^{2,\infty}$ continuity of the function $\sigma$ we easily obtain
				\begin{align*}
				\frac M\delta\mathrm{tr}(\Sigma D^2\psi)\le c\, M\frac{\psi(\bx,\by,\bz)}{\delta}\spazio.
				\end{align*}
				A straightforward computation shows us that
				\begin{align*}
				D^2&\left(\log(d(x)d(y)d(z))\right)_{|(x,y,z)=(\bx,\by,\bz)}=\\
				=&\begin{pmatrix}
				\frac{D^2d(\bx)}{d(\bx)}-\frac{Dd(\bx)\otimes Dd(\bx)}{d^2(\bx)} & 0 & 0\\
				0 & \frac{D^2d(\by)}{d(\by)}-\frac{Dd(\by)\otimes Dd(\by)}{d^2(\by)} & 0\\
				0 & 0 & \frac{D^2d(\bz)}{d(\bz)}-\frac{Dd(\bz)\otimes Dd(\bz)}{d^2(\bz)}
\end{pmatrix}\spazio.
\end{align*}
Then, we have
\begin{align*}
				&-\gamma \mathrm{tr}\left(\Sigma D^2\left(\log(d(x)d(y)d(z))\right)_{|(x,y,z)=(\bar{x},\bar{y},\bar{z})}\right)=\\
=&-\frac{\gamma}{d(\bar x)} \left( \mathrm{tr}(a(\bx)D^2d(\bx))- \frac{a(\bar x)Dd(\bar x)\cdot Dd(\bx) }{d(\bx)} \right) 
\\ & -\frac \gamma{d(\by)} \left(  \mathrm{tr}(a(\by)D^2d(\by))- \frac{a(\by)Dd(\by)\cdot Dd(\by)}{d(\by)}\right) \\&-\frac\gamma{d(\bz)}\left(  \mathrm{tr}(a(\bz)D^2d(\bz))-\frac{a(\bz)Dd(\bz)\cdot Dd(\bz)}{d(\bz)}\right) \,.
\end{align*}
Now we have to analyze the Hamiltonian term. As before, we need to split the computation in two parts, the first one including only the $\psi$ function and the last one involving the logarithmic term.
First of all, we recall that
$$
\hat{H}(\bt,\bx,\by,\bz)=\tilde{H}(\bt,\bx,D_x\varphi(\bx,\by,\bz))+\tilde{H}(\bt,\by,D_y\varphi(\bx,\by,\bz))-2\tilde{H}\left(\bt,\bz,-\miezz D_z\varphi(\bx,\by,\bz)\right).
$$
				Since
				\begin{align*}
				D_x\varphi(\bx,\by,\bz)=\frac M\delta D_x\psi(\bx,\by,\bz)-\gamma D\log d(\bx)
				\end{align*}
				and the same holds for $D_y\varphi$, $D_z\varphi$, we can write
				\begin{align*}
				\hat{H}(\bt,\bx,\by,\bz)=\spazio&\tilde{H}\left(\bt,\bx,\frac M\delta D_x\psi\right)-\gamma\int_0^1 \tilde{H}_p(\bt,\bx,p_1(\lambda))\frac{Dd(\bx)}{d(\bx)}d\lambda\spazio+\\
				+\spazio&\tilde{H}\left(\bt,\by,\frac M\delta D_y\psi\right)-\gamma\int_0^1 \tilde{H}_p(\bt,\by,p_2(\lambda))\frac{Dd(\by)}{d(\by)}d\lambda\spazio-\\
				-2&\tilde{H}\left(\bt,\bz,-\frac M{2\delta} D_z\psi\right)-\gamma\int_0^1 \tilde{H}_p(\bt,\bz,p_3(\lambda))\frac{Dd(\bz)}{d(\bz)}d\lambda\spazio,
				\end{align*}
				where
				\begin{align*}
				&p_1(\lambda)=\frac M\delta D_x\psi(\bx,\by,\bz)-\lambda\gamma D\log d(\bx)\spazio,\\
				&p_2(\lambda)=\frac M\delta D_y\psi(\bx,\by,\bz)-\lambda\gamma D\log d(\by)\spazio,\\
				&p_3(\lambda)=-\frac M{2\delta} D_z\psi(\bx,\by,\bz)+\miezz\lambda\gamma D\log d(\bz)\spazio.\\
				\end{align*}
				Putting these estimates in \eqref{apfuma}, one finds
\begin{align*}
k\varphi(\bx,\by,\bz)\le\spazio &C(1+M)\left(\delta+\frac{\psi(\bx,\by,\bz)}{\delta}\right)
\\
-&\tilde{H}\left(\bt,\bx,\frac M\delta D_x\psi\right)-\tilde{H}\left(\bt,\by,\frac M\delta D_y\psi\right)+2\tilde{H}\left(\bt,\bz,-\frac M{2\delta} D_z\psi\right) \\
-&\frac\gamma{d(\bx)}\int_0^1\left(\mathrm{tr}(a(\bx)D^2d(\bx)) - \frac{a(\bar x)Dd(\bar x)\cdot Dd(\bx) }{d(\bx)}-\tilde{H}_p(\bt,\bx,p_1(\lambda))Dd(\bx)\right)d\lambda \\-&\frac\gamma{d(\by)}\int_0^1\left(\mathrm{tr}(a(\by)D^2d(\by))
- \frac{a(\bar y)Dd(\bar y)\cdot Dd(\by) }{d(\by)}-\tilde{H}_p(\bt,\by,p_2(\lambda))Dd(\by)\right)d\lambda \\-&\frac\gamma{d(\bz)}\int_0^1\left(\mathrm{tr}(a(\bz)D^2d(\bz))- \frac{a(\bar z)Dd(\bar z)\cdot Dd(\bz) }{d(\bz)}-\tilde{H}_p(\bt,\bz,p_3(\lambda))Dd(\bz)\right)d\lambda\spazio.
\end{align*}
We use the invariance condition \rife{invariance} to get rid of the latter terms. So the inequality becomes
\begin{equation}\begin{split}\label{comfortably}
&k\varphi(\bx,\by,\bz)\le\spazio C(1+M)\left(\delta+\frac{\psi(\bx,\by,\bz)}{\delta}\right) \\-&\tilde{H}\left(\bt,\bx,\frac M\delta D_x\psi\right)-\tilde{H}\left(\bt,\by,\frac M\delta D_y\psi\right)+2\tilde{H}\left(\bt,\bz,-\frac M{2\delta} D_z\psi\right)+ c\, \gamma\,.
\end{split}\end{equation}
Finally, since $-D_z\psi= D_x \psi+ D_y \psi$, we use \eqref{hypH} to estimate the last terms involving $\tilde{H}$:
\begin{align*}
-&\tilde{H}\left(\bt,\bx,\frac M\delta D_x\psi\right)-\tilde{H}\left(\bt,\by,\frac M\delta D_y\psi\right)+2\tilde{H}\left(\bt,\bz,-\frac M{2\delta} D_z\psi\right)\le
\\
\le&
C_0    ( |x-z|^2+ |y-z|^2+ |x+y-2z|) (1+ \frac M\delta\, |D_x \psi+ D_y\psi |) 
\\
& \qquad 
 + C_1 |x-y | \frac M\de\left|D_x \psi- D_y\psi \right|\,.
\end{align*}
Using the precise values of $D_x\psi$ and $D_y\psi$, we estimate thanks to Young's inequality
\begin{align*}
C_1\, |x-y | \frac M\de\left|D_x \psi- D_y\psi \right| & \leq 
C\, |x-y | \frac M\de\left( |x-z|^3+ |y-z|^3 \right) 
\\
& 
 \leq C\, \frac M\de \psi
\end{align*}
and  similarly
\begin{align*}
& C_0    ( |x-z|^2+ |y-z|^2+ |x+y-2z|) (1+ \frac M\delta\, |D_x \psi+ D_y\psi |) \leq C \left(\sqrt \psi + \frac M\de \psi\right)
\\
& \leq C\left(\de + \frac M\de \psi\right)\,.
\end{align*}
Eventually,  we end up with
\be\label{estH}
-\tilde{H}\left(\bt,\bx,\frac M\delta D_x\psi\right)-\tilde{H}\left(\bt,\by,\frac M\delta D_y\psi\right)+2\tilde{H}\left(\bt,\bz,-\frac M{2\delta} D_z\psi\right)\le C\left(\de + \frac M\de \psi\right)\,.
\ee
Since
\begin{align*}
k\varphi(\bx,\by,\bz)=&kM\conc-k\gamma\log(d(\bx)d(\by)d(\bz))\ge\\\ge& kM\conc
\spazio,
\end{align*}
we obtain from \eqref{comfortably}--\rife{estH}
\begin{align*}
kM\conc\le C(1+M)\conc+C\, \gamma\spazio.
\end{align*}
We choose $k$ such that $kM-C(1+M)\geq 1$ to have
\begin{align*}
\delta\le\conc\le C\, \gamma\spazio.
\end{align*}
Choosing $\gamma$ sufficiently small we obtain a contradiction.

So, we have found that for $k$ sufficiently large and $\gamma$ sufficiently small, we have
\begin{align*}
v(t,x)+v(t,y)-2v(t,z)\le M\con-\gamma\log(d(x)d(y)d(z))\,,
\end{align*}
for every $x,y,z\in\Omega$, and any $\delta>0$. Letting $\gamma\to0$, we obtain
				\begin{align*}
				u(t,x)+u(t,y)-2u(t,z)\le e^{kT}M\conc\spazio.
				\end{align*}
				So, \eqref{semicu} is proved and the proof is concluded.
\end{proof}

\begin{rem}
We note that the hypothesis  \rife{hypH} is satisfied at least for classical Bellman equations where 
$$
H(x,p)= \sup_{\alpha \in A} (-b(x,\alpha)\cdot p - L(x,\alpha))
$$
assuming that $L(\cdot,\alpha)$  and $b(\cdot,\alpha)$ are $W^{2,\infty}$,  and both conditions hold uniformly with respect to $\alpha\in A$.
Indeed, we have
\begin{align*}
& 2\left(-b(z,\alpha)\cdot \left(\frac{p+q}2\right) - L(z,\alpha)\right) = (-b(z,\alpha)\cdot p - L(x,\alpha))
\\
& \qquad + (-b(z,\alpha)\cdot q - L(y,\alpha))  
+ \left(L(x,\alpha)+L(y,\alpha)-2L(z,\alpha) \right)
\\
& = (-b(x,\alpha)\cdot p - L(x,\alpha))+(-b(y,\alpha)\cdot q - L(y,\alpha))  
\\
& \quad + \frac12(b(x,\alpha)-b(y,\alpha))\cdot (p-q)  + \frac12(b(x,\alpha)+b(y,\alpha)- 2b(z,\alpha)) (p+q)
\\
& \quad + \left(L(x,\alpha)+L(y,\alpha)-2L(z,\alpha) \right)
\\
& \leq H(x,p)+ H(y,q) 
+ C\, |x-y|\, |p-q|
\\
& \qquad + C \, (|x-z|^2+ |y-z|^2+|x+y-2z| )  \,  |p+q|
\\
& \qquad + \left(L(x,\alpha)+L(y,\alpha)-2L(z,\alpha) \right)\,.
\end{align*}
Hence, taking the $sup_{\alpha}$ in the left-hand side, and using the regularity of $L$, implies  \rife{hypH}.
\end{rem}
		
%
%
Finally, we show with the next result that the semiconcavity of $u$ leads to the boundedness of the function $m$.
\begin{prop}
Assume that  the hypotheses of Theorem \ref{usemiconc} are satisfied, and, in addition,  that $H_p(t,x,p)\in W^{1,\infty}(Q_T\times K)$ for all compact sets $K\subset \R^N$. Suppose   that  the following condition holds near the boundary: there exists $\de_0$ such that 
\begin{align}\label{condizionecheboh}
\left(\tilde{b}(x)+H_p(t,x,p)\right)\cdot Dd(x)\le 0\spazio,\hspace{1cm}\forall x\in\Gamma_{\de_0},\ \forall p\in\R^N\spazio.
\end{align}
Then, if $(u,m)$ is a solution of \eqref{mfg}, we have $m\in L^\infty([0,T]\times\Omega)$.
\begin{proof}
We are going to apply the comparison principle in the equation of $m$. To do so, we call $\mu_\eps$ the solution of the following problem:
\begin{equation*}
\begin{cases}
(\mu_\eps)_t-\mathrm{div}(a(x) D\mu_\eps)-\mathrm{div}((\tilde{b}(x)+H_p(t,x,Du))\mu_\eps)=0\,,\qquad (t,x)\in (0,T)\times \Omega_\vep\\
\mu_\eps(0)=m_0\\
[a(x) D\mu_\eps+(\tilde{b}(x)+H_p(t,x,Du))\mu_\eps]\cdot\nu_{|\partial\Omega_\vep}=0\,,
\end{cases} 
\end{equation*}
where, as before, $\Omega_\vep=\{x\,:\, d(x) >\vep\}$. 
With the same arguments used previously we obtain
$$
\mu_\eps\to m\hspace{1cm}a.e.\mbox{ in }[0,T]\times\Omega\spazio.
$$
Since $u$ is a semiconcave function, we can split the last divergence term in order to get
\begin{equation*}
\begin{cases}
(\mu_\eps)_t-\mathrm{div}(a(x) D\mu_\eps)-(\tilde{b}+H_p(t,x,Du))D\mu_\eps+c(t,x)\mu_\eps=0\,,\qquad (t,x)\in (0,T)\times \Omega_\vep\\
\mu_\eps(0)=m_0\\
[a(x) D\mu_\eps+(\tilde{b}(x)+H_p(t,x,Du))\mu_\eps]\cdot\nu_{|\partial\Omega_\vep}=0\,,
\end{cases} 
\end{equation*}
where $c$ is defined as follows:
\begin{align*}
c(t,x)=-\mathrm{div}(\tilde{b})-\mathrm{tr}(H_{px}(t,x,Du))-\mathrm{tr}(H_{pp}(t,x,Du)D^2u)\spazio.
\end{align*}
Recall that in $\Omega_\vep$ the matrix $a(x)$ is elliptic, so  $u$ enjoys the standard parabolic regularity and $c(t,x)$ is well defined (at least in Lebesgue spaces).  We now estimate the function $c$. Since $H_{px}(t,x,Du)$ is in $L^\infty(Q_T)$ (because $u$ is globally Lipschitz) and $D^2u\le CI$, we get, up to changing $C$,
\begin{align*}
c(t,x)\ge -C-\mathrm{tr}(H_{pp}(x,Du)(D^2 u-CI))-C\mathrm{tr}(H_{pp}(x,Du))\ge -k\spazio,
\end{align*}
for a certain $k>0$ and since $H_{pp}(D^2u-CI)$ is a negative semi-definite matrix.\\
Calling $\mu_\eps^k=e^{-kt}\mu_\eps$, we have that $\mu_\eps^k$ is the solution of the following equation
\begin{align*}
\begin{cases}
(\mu_\eps^k)_t-\mathrm{div}((a(x)+\eps I)D\mu_\eps^k)-(\tilde{b}+H_p(x,Du))D\mu_\eps^k+(k+c(t,x))\mu_\eps^k=0\\
\mu_\eps^k(0)=m_0\\
[\eps D\mu_\eps^k+(\tilde{b}(x)+H_p(x,Du))\mu_\eps^k]\cdot\nu_{|\partial\Omega_\vep}=0
\end{cases}\spazio.
\end{align*}
We choose $k$ such that $k+c(t,x)\ge0$ for each $(t,x)\in[0,T]\times\Omega_\vep$. Then, thanks to \eqref{condizionecheboh}, it is immediate to prove that $M$ is a super-solution of the equation of $\mu_\eps^k$, for $M\ge\norm{m_0}_\infty$.\\
Now we can easily conclude the proof: thanks to the comparison principle, we have
\begin{align*}
\mu_\eps^k(t,x)\le M\implies\mu_\eps(t,x)\le e^{kT}M\overset{\eps\to 0}{\implies}m(t,x)\le e^{kT}M\spazio,
\end{align*}
where the estimates are true almost everywhere in $(t,x)$. Since $m\ge0$, the proof is concluded.
\end{proof}
\end{prop}
\section{Non-smooth domains}
Unfortunately, in many applications one needs to consider that the state variable does not belong  to a  $\mathcal{C}^{2}$ domain. This implies that the distance function  from the boundary of $\Omega$ turns out not to be a $\mathcal{C}^2$ function, and the invariance condition \eqref{invariance} becomes meaningless. However,  a generalization of the results obtained so far is possible, in the following setting.

\begin{thm}\label{nonsmooth}
Suppose that $\exists\psi\in\mathcal{C}^2(\Omega)$ such that $\psi>0$ in $\Omega$, $\psi=0$ in $\partial{\Omega}$ and  the following inequality holds in a  neighborhood $V$ of $\partial\Omega$:
\begin{equation}\label{invnonsmooth}
\begin{split}
&\mathrm{tr}(a(x)D^2\psi(x))-H_p(t,x,p)D\psi(x)   \ge \frac{a(x) D\psi(x) \cdot D\psi(x)}{\psi(x)}- C \psi(x)\,
\\
&  \qquad  \hbox{$\forall\ p\in\mathbb{R}^N$ ,$\,\,\forall t\in[0,T]\,\,$  and a.e. $x\in V$}\,.
\end{split}
\end{equation}
Then, all the results of the previous sections remain true replacing \rife{invariance} with \rife{invnonsmooth}.
\end{thm}

All the proofs can be done in the same way, replacing $d$ by $\psi$ and the set $\Gamma_{\eps}$ by $\{\psi<\eps\}\cap\Omega$.\\

This generalization plays a crucial role in order  that the smoothness assumption  for the domain $\Omega$  be weakened. 
As an example, we define a class of non-smooth domains and we prove that hypothesis \eqref{invnonsmooth} is satisfied for those ones.

\begin{defn}
Let $\Omega\subseteq\R^N$. We say that $\Omega$ is a generalized $\mathcal{C}^{2}$ domain if $\exists n\in\N$ and a collection of sets $\{\Omega^i\}_{1\le i\le n}$ such that $\overline \Omega^i$ is a  compact domain  of class $\mathcal{C}^{2}$ and
$$
\Omega=\overset{n}{\underset{i=1}{\bigcap}}\,\Omega^i\spazio.
$$
\end{defn}
From now on, when we write $d(x)$ and $d_i(x)$, in the case of a generalized $\mathcal{C}^{2}$ domain, we mean respectively $d_{\Omega}(x)$ and $d_{\Omega^i}(x)$. 
Moreover, we will use the following notation:
$$
\Omega^i_\eps=\{ d_i(x)>\eps\}\cap\Omega\,,\qquad\Gamma^i_\eps=\{ d_i(x)<\eps \}\cap\Omega\,.
$$
\vspace{0cm}

We now show that Theorem \ref{nonsmooth} applies to generalized $\mathcal{C}^{2}$ domains.

\begin{prop}
Let $\Omega=\overset{n}{\underset{i=1}{\bigcap}}\,\Omega^i$ be a generalized $\mathcal{C}^{2}$ domain. Suppose that $\exists\ \delta, C_0>0$ s.t. $\forall\ p\in\mathbb{R}^N$, $\forall 1\le i\le n$ and $\forall\ x\in\Gamma_\delta^i$ the following inequality holds:
\begin{equation}\label{invarianza2}
\mathrm{tr}(a(x)D^2d_i(x))-H_p(x,p)Dd_i(x)\ge\frac{a(x)Dd_i(x)Dd_i(x)}{d_i(x)}-C_0d_i(x)\spazio.
\end{equation}
Then all the results of the previous sections remain true for the non-smooth domain $\Omega$.
\begin{proof}
We have to prove that condition \eqref{invnonsmooth} is satisfied for a certain $\mathcal{C}^2$ function $\psi$.\\
	To do that, we consider $\phi:[0,+\infty)\to\R$ a $C^{2}$ function such that $\phi(s)=s$ when $s\le\frac\delta2$, $\phi\equiv1$	for $s\ge  \delta$ and $\phi'(s)\ge0$. 
	Moreover, we require that $\phi(x)\ge x$ in $[0, \delta]$. This can be done for $\delta$ sufficiently small.\\
	
	We take $\psi(x)=\prod\limits_i\phi(d_i(x))$ and we prove that \eqref{invnonsmooth} is satisfied in $V=\bigcup\limits_{\delta}\Gamma^i_\delta$ for a certain $\delta>0$.
	
	From now on, we will write $\phi_i$, $\phi'_i$, $\phi''_i$ instead of $\phi(d_i(x))$, $\phi'(d_i(x))$, $\phi''(d_i(x))$ to simplify the notation.	 Computing the derivative $D\psi$ and $D^2\psi$ we find
	\begin{align*}
	&D\psi=\mathlarger{\sum\limits_i}\prod\limits_{j\neq i}\phi_j\,\phi'_i Dd_i\,,\\
	&D^2\psi=\mathlarger{\sum\limits_i}\prod\limits_{j\neq i}\phi_j\,\phi'_i D^2d_i+\mathlarger{\sum\limits_i}\prod\limits_{j\neq i}\phi_j\,\phi''_i Dd_i\otimes Dd_i+\mathlarger{\sum\limits_{i,k\neq i}}\prod\limits_{j\neq i,k}\phi_j\,\phi'_i\phi'_k Dd_i\otimes Dd_k\,.
	\end{align*}
	Plugging these computations in \eqref{invnonsmooth} we find
	\begin{align*}
	&\mathrm{tr}(a(x)D^2\psi(x))-H_p(t,x,p)D\psi(x) - \frac{a(x) D\psi(x) \cdot D\psi(x)}{\psi(x)} + C \psi(x)=\\=&\mathlarger{\sum\limits_i}\prod\limits_{j\neq i}\phi_j\,\phi'_i\left(\mathrm{tr}(a(x) D^2d_i)-H_p(x,p)Dd_i\right)+\mathlarger{\sum\limits_i}\prod\limits_{j\neq i}\phi_j\,\phi''_i\, a(x)Dd_i\cdot Dd_i+\\
	+&\mathlarger{\sum\limits_{i,k\neq i}}\prod\limits_{j\neq i,k}\phi_j\,\phi'_i\phi'_k\, a(x)Dd_i\cdot Dd_k-\frac{\mathlarger{\sum\limits_{i,k}}\prod\limits_{j\neq i}\phi_j\prod\limits_{l\neq k}\phi_l\,\phi'_i\phi'_k\, a(x)Dd_i\cdot Dd_k}{\prod\limits_i \phi_i}+C\psi(x).
	\end{align*}
	We start analyzing the first term. 
	Because of the presence of $\phi'_i$, we can study each term of the sum only in $\Gamma_{\delta}^i$. So, choosing $\delta$ such that \eqref{invarianza2} holds true, we get
	\begin{align*}
	&\mathlarger{\sum\limits_i}\prod\limits_{j\neq i}\phi_j\,\phi'_i\left(\mathrm{tr}(a(x) D^2d_i)-H_p(x,p)Dd_i\right)\ge \mathlarger{\sum\limits_i}\prod\limits_{j\neq i}\phi_j\,\phi'_i\left(\frac{a(x) Dd_i\cdot Dd_i}{d_i}-C_0 d_i\right)\,.
	\end{align*}
	Since $d_i\le\phi_i$ in $\Gamma_{ \delta}^i$, one has
	$$
	-C_0\mathlarger{\sum\limits_i}\prod\limits_{j\neq i}\phi_j\,\phi'_i d_i\ge -C_1\mathlarger{\sum\limits_i}\prod\limits_{j}\phi_j\ge -C_1\psi(x)\,,
	$$
	where $C_1$ is a constant depending on $C_0$ and $n$ that can change from line to line.\\
	
	Then we look at the third and the fourth terms. Since, for $i\neq k$,
	$$
	\frac{\prod\limits_{j\neq i}\phi_j\,\prod\limits_{l\neq k}\phi_l}{\prod\limits_i \phi_i}=\prod\limits_{j\neq i,k}\phi_j\,,
	$$
	then we have
\begin{align*}
& \mathlarger{\sum\limits_{i,k\neq i}}\prod\limits_{j\neq i,k}\phi_j\,\phi'_i\phi'_k\, a(x)Dd_i\cdot Dd_k-\frac{\mathlarger{\sum\limits_{i,k}}\prod\limits_{j\neq i}\phi_j\prod\limits_{l\neq k}\phi_l\,\phi'_i\phi'_k\, a(x)Dd_i\cdot Dd_k}{\prod\limits_i \phi_i}
\\ & =	-\frac{\mathlarger{\sum\limits_i}\left(\prod\limits_{j\neq i}\phi_j\right)^2{(\phi'_i)}^2 a(x)Dd_i\cdot Dd_i}{\prod\limits_i\phi_i}=-\mathlarger{\sum\limits_i}\prod\limits_{j\neq i}\phi_j\,\frac{{(\phi'_i)}^2}{\phi_i} a(x)Dd_i\cdot Dd_i\,.
\end{align*}
	Using these estimates, we obtain
	\begin{equation}\begin{split}\label{quasifinal}
	&\mathrm{tr}(a(x)D^2\psi(x))-H_p(t,x,p)D\psi(x) - \frac{a(x) D\psi(x) \cdot D\psi(x)}{\psi(x)} + C \psi(x)\ge\\\ge\,\,&\mathlarger{\sum\limits_i}\prod\limits_{j\neq i}\phi_j\, a(x)Dd_i\cdot Dd_i\left(\frac{\phi'_i}{d_i}-\frac{{(\phi'_i)}^2}{\phi_i}+\phi''_i\right)+(C-C_1)\psi(x)\,.
	\end{split}\end{equation}
	To conclude, we want to prove that
	$$
	\frac{\phi'_i}{d_i}-\frac{{(\phi'_i)}^2}{\phi_i}+\phi''_i\ge -C_2\phi_i
	$$
	for a certain constant $C_2$. This is equivalent to prove that,  for all $x\in\R^+$,
	$$
	\phi''(x)\phi(x)x-{(\phi'(x))}^2x+\phi'(x)\phi(x)\ge -C_2\phi^2(x)x\,.
	$$
	Since  $\phi(s)=s$ in $[0,\frac\delta2]$, we obtain immediately that the left hand side term vanishes in this interval, and so the relation is verified. A similar computation occurs for $s\geq \de$. Finally, for $s\in [\frac\delta2,\delta]$  the relation is certainly satisfied for a constant $C_2$ depending on $\delta$. Therefore we obtain from \eqref{quasifinal}
	\begin{equation*}
	\mathrm{tr}(a(x)D^2\psi(x))-H_p(t,x,p)D\psi(x) - \frac{a(x) D\psi(x) \cdot D\psi(x)}{\psi(x)} + C \psi(x)\ge(C-C_1-C_2)\psi(x)\,,
	\end{equation*}
	which, for $C$ sufficiently large, proves that condition \eqref{invnonsmooth} is satisfied. This concludes the proof.
\end{proof}
\end{prop}

\section{Appendix}

In this Appendix, we give the proof of  two technical results.

\vskip1em

{\bf Proof of  Lemma \ref{compact-hjb}.} \quad 
We consider the sequence of compact sets $\{D_k\}_{k\in\mathbb{N}}$ defined in \rife{qn}.  For each $k\in\mathbb{N}$ we take a cut-off function $\xi_k$ such that
\be\label{cutoff}
\begin{cases}
\xi_k\in C_c^\infty(\Omega)\,, \quad 0\le\xi_k\le 1 & \\
\xi_k(x)\equiv 1 \quad \hbox{ for $x\in D_k$} & \\
\xi_k(x)\equiv 0 \quad \hbox{for $x\in D_{k+1}$.}
\end{cases}
\ee
For $\eps$ small enough and $\lambda>0$, we multiply the equation \eqref{hjbeps} by $e^{\lambda u_\eps}\xi_k^2$ and we integrate in $[t,T]\times\Omega_\eps$:
\begin{equation*}\begin{split}
&\frac 1\lambda\intok{k+1}e^{\lambda u_\eps}(t)\xi_k^2\spazio dx +\lambda\intfk{k+1}e^{\lambda u_\eps} a_\eps(x)Du_\eps \cdot Du_\eps\, \xi_k^2\spazio dxdt \\+&\intfk{k+1}H_\eps(t,x,Du_\eps)e^{\lambda u_\eps}\xi_k^2\spazio dxdt+\intfk{k+1} a_\eps(x)Du_\eps \cdot D\xi_k\,2e^{\lambda u_\eps}\xi_k\spazio dxdt \\
&\qquad\qquad = \frac 1\lambda\intok{k+1}e^{\lambda u_\eps}(T)\xi_k^2\,dx\le C 
\end{split}\end{equation*}
since $u_\vep$ is uniformly bounded. From the local uniform coercivity of $a_\vep$, we have $a_\vep(x) \geq \lambda_{k+1}\, I$ for $x\in D_{k+1}$. Using also the local uniform natural growth assumed upon $H_\vep$, and a local bound on  $a_\vep$, we deduce that
\begin{align*}
 \lambda\, \lambda_{k+1}\intfk{k+1} e^{\lambda u_\eps}\, |Du_\eps|^2\, \xi_k^2\spazio dxdt  \leq C_k \intifk{k+1}e^{\lambda u_\eps}(1+|Du_\eps|^2) \xi_k^2 \,dxdt 
\end{align*}
for some constant $C_k$ only depending on $k$. Choosing $\lambda$ sufficiently large (depending on $k$), we can bound the gradient of $u_\vep$ in  $D_k$. Hence, together with the $L^\infty$ bound, we deduce that  $u_\eps$  is bounded in $L^2([0,T];W^{1,2}(D_k))$ for each $k\in\mathbb{N}$. 

From \eqref{hjbeps} now we get that ${(u_\eps)}_t$ is bounded in $L^2([0,T];W^{-1,2}(D_k))$. So, by   \cite[Corollary 4]{Simon}, we deduce that $u_\vep$ is relatively compact in $L^2(D_k)$. By a  standard diagonal argument, we  can therefore extract a  subsequence, which we still denote by $u_\vep$, such that 
\begin{equation*}
u_\eps\to u\hspace{0.1cm}\mbox{weakly in $L^2([0,T];W^{1,2}(K))$ and strongly in $L^p([0,T]\times K)$ for every $p<\infty$,} 
\end{equation*}
for any compact subset $K\subset \Omega$. 

We now aim at getting  the strong convergence. To this purpose we assume that the matrix $a_\vep(x)$ converges (up to subsequences) almost everywhere in $\Omega$ towards some matrix $a(x)$. We further suppose by now that $u_\vep(T)$ converges almost everywhere in $\Omega$ to some function $g(x)$, in order to get the full convergence up to $t=T$. We notice that, since $u_\vep$ is uniformly bounded, this implies that $u_\vep(T) \to g$ strongly in $L^p(\Omega)$ for all $p<\infty$. Moreover,  since ${(u_\eps)}_t$ converges to $u_t$ weakly in  $L^2([0,T];W^{-1,2}(D_k))$ for all $D_k$, one has  that $u_t \in L^2([0,T];W^{-1,2}(D_k))$, so $u\in C^0([0,T]; L^2(D_k))$ and actually $u(T)= g(x)$ in $\Omega$.

Now we multiply  \rife{hjbeps} by  $\psi(u_\eps-u)  \xi_k^2$, for a convenient  increasing function $\psi$ to be chosen later. We proceed in a  similar way as above obtaining
\begin{align*}
&   \lambda_{k+1}\intifk{k+1}\psi'(u_\eps-u)\, |Du_\eps- Du|^2\, \xi_k^2\spazio dxdt  
\\
& \quad \leq - \intifk{k+1} \psi'(u_\vep-u) a_\vep(x)Du\cdot (Du_\vep-Du)\, \xi_k^2
\\
& \quad + C_k \intifk{k+1} |\psi(u_\eps-u)| (1+|Du_\eps|^2)  \, \xi_k^2 \,dxdt  - 
\int_0^T \langle \partial_t u_\vep, \psi(u_\vep-u)\xi_k^2\rangle\,,
\end{align*}
which yields, using $|Du_\eps|^2\leq 2(|Du_\eps- Du|^2+ |Du|^2)$:
\begin{align*}
&   \intifk{k+1} [\lambda_{k+1}\psi'(u_\eps-u)- 2C_k|\psi(u_\eps-u)|]\, |Du_\eps- Du|^2\, \xi_k^2\spazio dxdt  \leq 
\\
& \leq - \intifk{k+1} \psi'(u_\vep-u) a_\vep(x)Du\cdot (Du_\vep-Du)\, \xi_k^2
\\
& \quad  + C_k \intifk{k+1} |\psi(u_\eps-u)| (1+2 |Du|^2)  \, \xi_k^2 \,dxdt  + 
\int_0^T \langle \partial_t u_\vep, \psi(u_\vep-u)\xi_k^2\rangle\,.
\end{align*}
Now we choose $\psi$ such that  $s\psi(s)\ge 0$ and
$\lambda_{k+1}\psi'(s)-2C_k |\psi(s)| > 0$ for all $ s\in\R$.
A typical choice is e.g.  $\psi(s)=se^{bs^2}$ with $b=\frac{C_k^2}{\lambda_{k+1}^2}$. So we get
\be\label{espk}
\begin{split}
&   \intifk{k+1}  |Du_\eps- Du|^2\, \xi_k^2\spazio dxdt  \leq 
\\
& \leq - C_k \intifk{k+1} \psi'(u_\vep-u) a_\vep(x)Du\cdot (Du_\vep-Du)\, \xi_k^2
\\
& \quad  + C_k \intifk{k+1} |\psi(u_\eps-u)| (1+2 |Du|^2)  \, \xi_k^2 \,dxdt  + 
C_k \int_0^T \langle \partial_t u_\vep, \psi(u_\vep-u)\xi_k^2\rangle\,,
\end{split}
\ee
where $\langle \cdot,\cdot\rangle$ is the duality between $W^{1,2}(\Omega_\vep)$ and its dual. We conclude by showing that all terms in the right-hand side go  to zero as $\vep\to0$. Indeed, 
if $\Psi(s)= \int_0^s \psi(r)dr$, we have
\begin{align*}
& \int_0^T \langle \partial_t u_\vep, \psi(u_\vep-u)\xi_k^2\rangle = \into \Psi(u_\vep(t)-u(t))dx \mathop{I}^T_0 + \int_0^T \langle \partial_t u, \psi(u_\vep-u)\xi_k^2\rangle
\\
& \leq \into \Psi(u_\vep(T)-u(T))dx  + \int_0^T \langle \partial_t u, \psi(u_\vep-u)\xi_k^2\rangle\, dt
\end{align*}
and the last two terms converge  to zero because  $\psi(u_\vep-u)\xi_k^2$ weakly converges to zero in $L^2([0,T];W^{1,2}(D_k))$ and $u_\vep(T)\to g(x)=u(T)$ almost everywhere (hence $\Psi(u_\vep(T)-u(T))\to 0$ in $L^1$ by dominated convergence). Still using Lebesgue's dominated convergence theorem, we have that $|\psi(u_\eps-u)| (1+2 |Du|^2)\to 0$ in $L^1((0,T)\times D_{k+1})$.
Finally, using that $a_\vep(x)$ converges almost everywhere, we can deduce that $\psi'(u_\vep-u) a_\vep(x)Du$ converges strongly in $L^2((0,T)\times D_{k+1})$, and since $Du_\vep$ converges weakly to $Du$, the first integral in the right-hand side of \rife{espk} converges to zero as well. Therefore, \rife{espk} implies that
$$
\intifk{k+1}  |Du_\eps- Du|^2\, \xi_k^2\spazio dxdt  \mathop{\to}^{\vep\to 0} 0
$$
which implies the strong convergence of $u_\vep$ to $u$ in $L^2([0,T];W^{1,2}(D_k))$ for all $D_k$. 

We finish by noticing that, in case $u_\vep(T)$ is not assume to converge strongly, then the above argument needs to be localized, which means using the test function $\psi(u_\eps-u)  \xi_k^2\, (T-t)$. With the same arguments as above, in that case one concludes the strong convergence in $L^2([0,t];W^{1,2}(D_k))$ for all $t<T$, but not up to $t=T$.
\qed

\vskip2em
We now provide the proof of the analog compactness result for the Fokker-Planck equation. 

{\bf Proof of Lemma \ref{compact-fp}.} \quad We only sketch the proof since this is just  a local version of compactness results which are well-known in the case of   boundary value problems.  

We first obtain local estimates as in \cite[Lemma 3.3]{Pumi}:  namely, for any $q<\frac{N+2}{N+1}$,
\be\label{stime_loc}
\intifk{k}|Dm_\eps|^q\spazio dxdt\le C_{q,k}
\ee 
for every $D_k$ defined in \rife{qn}. This estimate in particular implies that $m_\vep$ is bounded in $L^q(0,T; W^{1,q}(K))$ for $q<\frac{N+2}{N+1}$, for any compact subset $K$.
From the equation and the local boundedness of $a_\vep$, $b_\vep$, we deduce that $(m_\eps)_t$ is bounded in $L^q(0,T; W^{-1,q}(K))$, so applying standard compactness results (see \cite[Corollary 4]{Simon}) we get that $m_\vep$ is relatively compact in $L^1(0,T; L^1(K))$. Through a diagonal procedure, we can extract a  subsequence, which is not relabeled, such that $m_\vep $ converges almost everywhere in $Q_T$, and in $L^1(0,T; L^1(K))$ for every compact subset $K\subset \Omega$, towards some function $m$ which actually belongs to $L^\infty((0,T);L^1(\Omega))$ because of Fatou's lemma and the fact that $\|m_\vep(t)\|_{\elle1}$ is uniformly bounded.

In order to obtain a strong convergence in $C^0([0,T];L^1_{loc}(\Omega))$, we use some renormalization argument similar as in \cite[Theorem 6.1]{P-Arma}. To this purpose, we use the auxiliary function
\be\label{sn}
S_n(r)=n\,S\left(\frac rn\right)\,,\,\hbox{where $S(r)=\int_0^r S'(r)dr$,} \quad S'(r)=\begin{cases} 1 & \hbox{if $|s|\leq 1$}\\
2-|s| & \hbox{if $1<|s|\leq 2$}
\\
0 & \hbox{if $|s|>2$} \end{cases} 
\ee
so that 
$S_n$ is a sequence of bounded functions which  converges to the identity locally uniformly as $n\to \infty$. 
Let $\xi_k$ be the cut-off function  defined   in \eqref{cutoff}. By choosing $(1-S_n'(m_\vep)) \xi_k^2$ as test function in \rife{meps} one obtains
\begin{align*} &
\into (m_\vep-S_n(m_\vep))(t)\xi_k^2\, dx  +\lambda_{k+1}\, \frac1n\int_0^t \int_{D_{k}} | Dm_\vep|^2\, \mathds 1_{\{n<m_\vep<2n\}} dxds
\\
& \qquad \leq  C_k \left\{\int_0^t \int_{D_{k+1}} m_\vep\, \mathds 1_{\{n<m_\vep \}} dxds + \int_0^t \int_{D_{k+1}} |D m_\vep  | \mathds 1_{\{n<m_\vep \}} dxds \right\}
\\
& \qquad \qquad + \into (m_\vep(0)-S_n(m_\vep(0)))\,\xi_k^2 \, dx\,,
\end{align*}
where we used the local ellipticity of $a_\vep$, and the local boundedness of $a_\vep$ and $b_\vep$.
Since $0\leq r-S_n(r)\leq r\, \mathds 1_{\{r>n\}}$, and since $m_\vep(0)$ converges in $L^1(K)$ for all compact subsets $K$, last term is  small as $n\to \infty$, uniformly with respect to $\vep$. The same holds for the previous terms in the right-hand side due to the local bounds on $m_\vep$, $Dm_\vep$. Hence it holds
\be\label{fett}
\lim\limits_{n\to \infty} \,\, \sup_\vep \,\,\,\frac1n\int_0^T \into | Dm_\vep|^2\, \mathds 1_{\{n<m_\vep<2n\}} \, \xi^2\, dxdt = 0
\ee
and
\be\label{fett2}
\lim\limits_{n\to \infty} \,\, \sup_{\{\vep>0, t\in [0,T]\}}\,\,\, \into (m_\vep-S_n(m_\vep))(t)\xi^2\, dx  =0
\ee
for any cut-off function $\xi$.

Now one can renormalize the equation for $m_\vep$. Indeed, thanks to  \rife{fett} the function $S_n(m_\vep)$ satisfies
\be\label{ren}
(S_n(m_\eps))_t-\mathrm{div}(a_\vep^* (x)DS_n(m_\eps)+ b_\vep\, m_\vep\, S_n'(m_\vep))=R_{\vep,n}\quad\mbox{in }(0,T)\times\Omega_\eps\,
\ee
where $R_{\vep, n}$ is such that 
\be\label{reps}
\lim\limits_{n\to \infty} \,\, \sup_\vep\,\,\, \int_0^T \into |R_{\vep,n}|\, \xi^2\, dxdt =0\,.
\ee
Consider now a sequence  $\{m^n_j\}$ in $L^2(0,T; W^{1,2}_{loc}(\Omega))$ approximating the function $S_n(m)$ with the following properties:
$$
\begin{cases}
\partial_t m^n_j  = -j (m^n_j- S_n(m))\,,\quad \qquad \|m^n_j\|_\infty\leq n 
\\
\m
m^n_j\mathop{\to}\limits^{j\to \infty} S_n(m)\quad \hbox{in $L^2(0,T; W^{1,2}(K)) $,} \qquad 
m^n_j(0) \mathop{\to}\limits^{j\to \infty} S_n(m_0) \quad \hbox{in $L^1(K)$,}
& 
\end{cases}
$$
for any  compact  $K\subset \Omega$. We take $T_1(S_n(m_\vep)-m^n_j)\, \xi^2$ as test function in \rife{ren} and we 
obtain (we denote $\Theta_1(r)=\int_0^r T_1(s)ds$):
\begin{align*} &
\into \Theta_1\left(S_n(m_\vep)-m^n_j \right)(t)\, \xi^2dx + C \int_0^t \into |DT_1(S_n(m_\vep)-m^n_j)|^2\, \xi^2 \\
& \leq  \into \Theta_1\left(S_n(m_\vep(0))-m^n_j (0)\right)\, \xi^2dx
 - \int_0^t \into b_\vep \, m_\vep \, S_n'(m_\vep)D T_1(S_n(m_\vep)-m^n_j)\, \xi^2
\\
& \quad - \int_0^t \into a_\vep^*(x) \, Dm_j^n\, DT_1(S_n(m_\vep)-m^n_j)\, \xi^2 \, dxds-
\int_0^t \into (m^n_j)_t T_1(S_n(m_\vep)-m^n_j) \, \xi^2dxds 
\\
& \quad 2 \int_0^t \into [a_\vep^*(x)\, DS_n(m_\vep)+ b_\vep\, m_\vep \, S_n'(m_\vep)] D\xi\, T_1(S_n(m_\vep)-m^n_j)\,\xi\, dxds
\\
& \quad + \sup_\vep\,\,\, \int_0^T \into |R_{\vep,n}|\, \xi^2\, dxdt\,.
\end{align*}
We will let first $\vep\to 0$ and then $j\to \infty$. Since all the integrals are localized in a compact subset of $\Omega$, we use the almost everywhere convergence of $m_\vep$ (hence $S_n(m_\vep)$ converges  in $L^1$) and the weak convergence of $DS_n(m_\vep)$ in $L^2$, as well as the convergences of $a_\vep(x)$ and $b_\vep(t,x)$. Then we find, as $\vep\to 0$:
\begin{align*} &
\limsup_{\vep \to 0}\into \Theta_1\left(S_n(m_\vep)-m^n_j \right)(t)\, \xi^2dx 
 \\ 
 & \leq  \into \Theta_1\left(S_n(m_0)-m^n_j (0)\right)\, \xi^2dx
 - \int_0^t \into b  \, m  \, S_n'(m) D T_1(S_n(m)-m^n_j)\, \xi^2
\\
& \quad - \int_0^t \into a^*(x) \, Dm_j^n\, DT_1(S_n(m)-m^n_j)\, \xi^2 \, dxds-
\int_0^t \into (m^n_j)_t T_1(S_n(m)-m^n_j) \, \xi^2dxds 
\\
& \quad 2 \int_0^t \into [a^*(x)\, DS_n(m)+ b\, m  \, S_n'(m)] D\xi\, T_1(S_n(m)-m^n_j)\,\xi\, dxds
\\
& \quad + \sup_\vep\,\,\, \int_0^T \into |R_{\vep,n}|\, \xi^2\, dxdt\,.
\end{align*}
Thanks to the properties of $m^n_j$ we have $(m^n_j)_t T_1(S_n(m)-m^n_j) \geq 0$. Hence that term can be dropped. For all other terms, we let $j\to \infty$;  using that $m^n_j$ converges to $S_n(m)$ and $m_j^n(0)$ to $S_n(m_0)$ we conclude 
that 
\be\label{jeps}
\limsup_{j\to \infty}\,\, \limsup_{\vep \to 0} \into \Theta_1\left(S_n(m_\vep)-m^n_j \right)(t)\, \xi^2dx  \leq  \omega(n)
\,, \qquad \omega(n):=\sup_\vep\,\,\, \int_0^T \into |R_{\vep,n}|\, \xi^2\, dxdt\,.
\ee
Last term will vanish as $n\to \infty$ due to \rife{reps}. 
Now we estimate
$$
\into |S_n(m_\vep)-S_n(m)|(t)\, \xi^2 \, dx   \leq \into |S_n(m_\vep)-m^n_j|(t)\, \xi^2dx + 
\into |S_n(m)-m^n_j|(t)\, \xi^2dx\,.
$$
Using that $|s| \leq C\, \max(\Theta_1(s), \sqrt{\Theta_1(s)})$, due to \rife{jeeps} and to the convergence of $m_j^n$ towards $S_n(m)$, we get after letting $\vep\to 0$ and $j\to \infty$:
$$
\limsup_{\vep \to 0} \, \into |S_n(m_\vep)-S_n(m)|(t)\, \xi^2 \, dx \leq C \omega(n) \mathop{\to}^{n\to \infty}0\,, 
$$
and the above holds uniformly with respect to $t\in [0,T]$.
Putting together this estimate with \rife{fett2}, we conclude that 
$$
\lim_{\vep \to 0}\,\,\, \sup\limits_{[0,T]}\into | m_\vep(t)-m |(t)\, \xi^2 \, dx =0\,.
$$
Hence $m\in C^0([0,T]; L^1_{loc}(\Omega))$ and the uniform convergence holds.
\qed

\end{document}